\documentclass[a4paper,10pt]{amsart}
\usepackage{amsthm,amsmath,amssymb,amsfonts}
\usepackage[T1]{fontenc}
\usepackage[utf8]{inputenc}
\usepackage{MnSymbol}
\usepackage{marvosym}
\usepackage[svgnames]{xcolor}
\definecolor{blue(munsell)}{rgb}{0.0, 0.5, 0.69}
\usepackage{caption}
\usepackage{enumitem}
\usepackage{mathtools}
\usepackage{epigraph}
\usepackage{color}
\usepackage{hyperref}
\hypersetup{
    colorlinks = true,
    linkbordercolor = {red},
   	linkcolor ={teal},
	anchorcolor = {pink},
	citecolor =  {orange},
	filecolor = {teal},
	menucolor = {teal},
	runcolor =  {teal},
	urlcolor = {teal},
}
\usepackage{cleveref}
\usepackage{trimclip}

\DeclareFontFamily{U}{min}{}
\DeclareFontShape{U}{min}{m}{n}{<-> udmj30}{}
\newcommand{\yo}{\!\text{\usefont{U}{min}{m}{n}\symbol{'210}}\!}
\newcommand{\yoco}{\text{\reflectbox{\yo}}}

\usepackage{stmaryrd}
\usepackage{tikz}
\usepackage{tikz-cd}
\usepackage{quiver}

\theoremstyle{definition}
\newtheorem{thm}{Theorem}[subsection]
\newtheorem*{thm*}{Theorem}
\newtheorem*{constr*}{Construction}
\newtheorem{prop}[thm]{Proposition}
\newtheorem{lem}[thm]{Lemma}
\newtheorem{cor}[thm]{Corollary}

\newtheorem{defn}[thm]{Definition}
\newtheorem{rem}[thm]{Remark}
\newtheorem{para}[thm]{}

\newtheorem{exa}[thm]{Example}

\newcommand{\Dbb}{\mathbb{D}}

\newcommand{\Ibb}{\mathbb{I}}

\newcommand{\Bcal}{\mathcal{B}}
\newcommand{\Ccal}{\mathcal{C}}
\newcommand{\Dcal}{\mathcal{D}}
\newcommand{\Ecal}{\mathcal{E}}
\newcommand{\Fcal}{\mathcal{F}}
\newcommand{\Gcal}{\mathcal{G}}

\newcommand{\Mcal}{\mathcal{M}}
\newcommand{\Ocal}{\mathcal{O}}
\newcommand{\Pcal}{\mathcal{P}}
\newcommand{\Scal}{\mathcal{S}}
\newcommand{\Tcal}{\mathcal{T}}

\newcommand{\Set}{\mathrm{Set}}

\newcommand{\op}{{}^{\mathrm{op}}}
\newcommand{\pt}{\mathsf{pt}}
\newcommand{\ev}{\mathsf{ev}}
\newcommand{\Ind}{\mathsf{Ind}}
\newcommand{\Pro}{\mathsf{Pro}}
\newcommand{\Coll}{\mathsf{Coll}}
\newcommand{\impr}[3]{{#1} \oslash_{#2} {#3}}
\DeclareMathOperator{\id}{id}

\DeclareMathOperator{\mor}{mor}
\DeclareMathOperator{\ob}{ob}
\DeclareMathOperator{\dom}{dom}
\DeclareMathOperator*{\colim}{colim}

\DeclareMathOperator{\PSh}{PSh}
\DeclareMathOperator{\Sh}{Sh}

\DeclareMathOperator{\Exp}{Exp} 

\usepackage[normalem]{ulem}

\title{Topoi with enough points}

\author{Ivan Di Liberti}
\author{Morgan Rogers}

\address{
Ivan \textsc{Di Liberti} \newline
Department of Philosophy, Linguistics and Theory of Science\newline
University of Gothenburg\newline
Gothenburg, Sweden\newline
\href{mailto:diliberti.math@gmail.com}{\sf diliberti.math@gmail.com}
}

\address{
Morgan \textsc{Rogers}: \newline
Laboratoire d'Informatique Paris Nord\newline
Université Sorbonne Paris Nord\newline
Paris, France\newline
\href{mailto:rogers@lipn.univ-paris13.fr}{\sf rogers@lipn.univ-paris13.fr}
}

\thanks{The first named author received funding from Knut and Alice Wallenberg Foundation, grant no.~2020.0199.}

\begin{document}

\maketitle

\begin{abstract}
We extend Deligne's original argument showing that locally coherent topoi have enough points, clarified using collage diagrams. We show that our refinement of Deligne's technique can be adapted to recover every existing result of this kind, including the most recent results about $\kappa$-coherent $\kappa$-topoi. Our presentation allows us to relax the cardinality assumptions typically imposed on the sites involved. We show that a larger class of locally finitely presentable toposes have enough points and that a closed subtopos of a topos with enough points has enough points.
 
\end{abstract}

 {
   \hypersetup{linkcolor=black}
   \tableofcontents
 }

\section*{Introduction}

\subsection*{General landscape}
This paper is concerned with the general problem of assessing whether a topos has enough points, with motivations coming both from geometry and logic. As we will see below, this problem has been influential in both fields up to the present day (see for example the work of Lurie \cite[VII, 4.1]{lurie2004derived} and Barwick et al. \cite[II.3]{barwick2018exodromy}, for the case of geometry, and Espíndola and Kanalas \cite{espindola2023every} for the case of logic). We shall start by giving an account of the development of this problem to properly frame our contribution and the significance of our work.

In 1964 Pierre Deligne proved a very celebrated theorem in topos theory.
\begin{thm*}[Deligne, {\cite[Exposé VI, 9.0]{bourbaki2006theorie}}]
    Every locally coherent topos has enough points.
\end{thm*}
The theorem's original motivation came from algebraic geometry, but after Joyal and Reyes developed the theory of classifying topoi \cite{reyes1974sheaves}, it was observed by Lawvere \cite{lawvere1975continuously} that Deligne's theorem was essentially the statement of Gödel completeness theorem for first order logic in disguise. This realisation crowned Deligne's theorem as a major result in categorical logic, and a source of inspiration for finding other completeness-like results using techniques from topos theory. 

To the present day Deligne's theorem remains the main argument to show that a wide class of topoi have enough points, and to some extent this paper investigates the limits (and the possibly unexploited potential) of this result. With current technology, we have essentially two different techniques to prove that (locally) coherent topoi have enough points. 

\begin{itemize}
    \item In \cite[7.4]{johnstone1977topos}, Johnstone gives an account of the original argument by Deligne, making a couple of minor adjustment for the sake of readability and elegance. 
    \item In \cite[IX.11]{sheavesingeometry}, Moerdijk and Mac Lane give a more conceptual proof of the theorem, based on an injectivity property of coherent topoi. This idea was further investigated by the first author in \cite{di2022geometry}.
\end{itemize}

These two proof strategies are -- to our understanding -- different in nature, and seem to rely on different aspects of coherent topoi, with Deligne's result clearly the one with a greater chance of being generalised to other classes of topoi. Following Deligne's theorem, it was some time before new results emerged showing that further classes of topoi have enough points. Makkai and Reyes proved that so-called \textit{separable} topoi have enough points. 
\begin{thm*}[Makkai and Reyes, {\cite[Theorem 6.2.4, page 180]{makkai2006first}}]
Let $(\Ccal,J)$ be a site where $\Ccal$ is a countable category with pullbacks and $J$ is generated by a countable family of sieves. Then $\Sh(\Ccal,J)$ has enough points.
\end{thm*}
This result was inspired by the Fourman-Grayson completeness theorem for the logic $L_{\omega_1,\omega_0}$ (see \cite{fourman1982formal}), and indeed it is almost the translation of it into topos-theoretic language through the bridge of classifying topoi. The proof in \cite{makkai2006first} is a bit sketchy and of model theoretic inspiration, thus it is hard to compare this result to Deligne's one. 

\subsection*{Recent developments}

Let us now recount the most recent developments in this topic. Lurie has imported Deligne's original argument to the world of $\infty$-topoi \cite[VII, 4.1]{lurie2004derived}. The proof carries over with minor adjustments under the additional assumption that the $\infty$-topos is hypercomplete (i.e.\ that it can be presented in terms of a site). Lurie's analysis did not make Deligne's argument sharper in any sense but stands as the best result available for $\infty$-topoi in the broad landscape of geometry. 

On the logical side, the main advancements are due to Espíndola \cite{ESPINDOLA2019137,espindola2020infinitary,espindola2023every}. The first two of these papers generalise the result due to Makkai and Reyes to what he calls $\kappa$-coherent $\kappa$-topos, a notion designed to deliver completeness results for $\kappa$-infinitary generalizations of coherent logic. His techniques rely heavily on logic, so the relationship between these results and that of Deligne remained unclear, especially in terms of technology. We recommend \cite{ESPINDOLA2019137} for a comprehensive reference of all the completeness result of infinitary logic and their relation to $\kappa$-topoi. More recently, together with Kanalas, they have provided a more categorical analysis of the original statement, which also delivers a vast generalization of the original results that Espíndola had achieved in his PhD thesis \cite{espindola2023every}.

\begin{thm*}[Espíndola and Kanalas, { \cite[Thm. 5.4]{espindola2023every}}]
Let $\kappa$ be a regular cardinal. Let $(\Ccal,J)$ be a $\kappa$-site, meaning that $\Ccal$ has $<\kappa$-limits and is of local size $\leq \kappa$ and that $J$ is generated by a family of presieves having cardinality $\leq \kappa$. Then $\Sh(\Ccal,J)$ has enough $\kappa$-points.
\end{thm*}

Espíndola and Kanalas' proof, while being categorical in nature, does not seem to follow Deligne's original strategy beyond inspiration. In parallel with these developments, the community revolving around topos theory has tried to understand the limits of Deligne's original argument and its possible generalization. Quite independently the authors of this paper and Tim Campion conjectured that every \textit{locally finitely presentable topos could have enough points}. Some evidence of discussion around this idea can be found in a collection of MathOverflow questions \cite{401429,338862,417780,371473,347107}. This conjecture finds its motivations in a number of examples, which we may list below: 

\begin{itemize}
    \item coherent topoi are locally finitely presentable (\cite[D3.3.12]{elephant2}), and the gap between these two notions seems quite small, at least by a superficial analysis,
    \item presheaf topoi are often not coherent and yet they are always locally finitely presentable, and they are the easiest example of topos with enough points,
    \item Hoffmann-Lawson duality \cite{hofmann1978spectral} from the theory of locales seems to support the conjecture that exponentiable topoi (which include locally finitely presentable topoi) have enough points.
\end{itemize}

In an early version of \cite{rogers2021supercompactly} the second author found a flawed proof of the conjecture (see the discussion in the comments of \cite{401429}; 
Marc Hoyois pointed out the error), and this collaboration stems from an attempt to correct that mistake. Before proceeding further with discussing our contribution, we can summarise that while we were able to prove that a large class of locally finitely presentable toposes have enough points, since the original preprint of this work was released a claimed counterexample to the general conjecture has emerged (see \Cref{rem:Marques}).

\subsection*{Our contribution and structure of the paper}


In Section~\ref{collages} we briefly recall the notion of collage of a profunctor and establish the diagrams featuring in the subsequent discussion, such as the following:
\[
\begin{tikzcd}[ampersand replacement=\&]
\& \cdot \\
p \&\& Y \&\& X \\
\& \cdot
\arrow["z"{description, pos=0.7}, from=2-1, to=2-3]
\arrow[dotted, no head, from=1-2, to=3-2]
\arrow["g"', shift right=2, from=2-3, to=2-5]
\arrow["f", shift left=2, from=2-3, to=2-5]
\end{tikzcd}
\]
Collages offer a convenient setting where points (on the left) can interact with objects (on the right) of a topos in a common category.

Section~\ref{sec:Deligne} is the main section of the paper. After introducing the notion of \textit{improvement} (\Cref{def:improves}), which is designed to isolate the central idea of Deligne's proof, we prove our core theorem (which we reproduce below in an easy-to-read form).

\begin{thm*}[\ref{thm:improvement}]
Let $j: \Fcal \rightarrowtail \Ecal$ be an inclusion of topoi. Suppose that every point $p$ can be improved. If $\Ecal$ has enough points, then $\Fcal$ has enough points.
\end{thm*}

Section~\ref{sec:recover} is devoted to recovering all the existing theorems concerning topoi with enough points using our technology. \Cref{cor:deligne} recovers the original theorem due to Deligne, \Cref{thm:supermakkai} implies Makkai and Reyes' result \Cref{cor:makkai}, while \Cref{thm:superespindola} generalizes Espindola and Kanalas' result \Cref{cor:espindola}. Our generalization involves introducing the notion of ($\kappa$-)\textit{warp}, mildly constrained data generating a Grothendieck topology (\Cref{defn:basis}) and the corresponding notion of ($\kappa$-)\textit{woven} site (Definitions \ref{defn:woven} and \ref{defn:kwoven}).

\begin{thm*}[\ref{thm:superespindola} and \ref{thm:pbfree}]
Let $\kappa$ be a regular cardinal. Let $(\Ccal,J)$ be a $\kappa$-woven site where $\Ccal$ is a category with limits of $<\kappa$-cochains. Then $\Sh(\Ccal,J)$ has enough $\kappa$-points\footnote{Points preserving limits of size $<\!\kappa$.}.
\end{thm*}

We devote a further subsection (\Cref{ssec:locales}) to discussing some localic implications of our work. There we recover Heckman's theorem concerning countably generated locales having enough points. 

We push the boundaries of our results further in Section~\ref{sec:extend}, first showing that the hypothesis of pullbacks can be eliminated from Theorems \ref{thm:supermakkai} and \ref{thm:superespindola}. This section is also where we give our best attempt to prove that locally finitely presentable topoi have enough points.

\begin{thm*}[\ref{cor:lfpcountable}]
A locally finitely presentable topos whose class of finitely presentable objects is essentially countable has enough points.
\end{thm*}

Finally we present a $2$-dimensional analysis of our core theorem (and of Deligne's argument in general). Besides offering a more conceptual interpretation of the theorem, we show that topoi with enough points are closed under closed embeddings:

\begin{thm*}[\ref{thm:closedinc}]
A closed subtopos of a topos with enough points has enough points.
\end{thm*}

\section{Collage diagrams} \label{collages}

This section recalls the notion of \textit{collage}, which we will employ as a technical framework in which to organise the proof of our core theorem (\ref{thm:improvement}). After the first definitions, we take the opportunity to tailor the discussion and present some examples in order to introduce the reader to our point of view on Deligne's theorem.

\subsection{Definitions}


\begin{para}[Evaluation pairings] \label{constr:evpair} \label{exa:presheaves}
    Recall that for $\Ecal$ a topos, we have an evaluation pairing
    \[\ev_{\Ecal} : \pt(\Ecal) \times \Ecal \to \Set\]
    which is defined by $(p,X) \mapsto p^*(X)$. This pairing will play an important role in the paper, so let us spell our more in detail a concrete example. Let $\Ecal = [\Ccal\op,\Set]$. Then $\pt(\Ecal)$ is (equivalent to) the ind-completion $\Ind(\Ccal\op)$. Throughout, we shall denote by
\begin{align*}
\yo: \Ccal &\to [\Ccal\op,\Set] \\
\yoco: \Ccal &\to \Pro(\Ccal) = \Ind(\Ccal\op)\op \subseteq [\Ccal,\Set]\op
\end{align*}
the Yoneda embedding and (a restriction of) its dual. Restricting along these embeddings, we find that $\ev(\yoco(C),\yo(D)) \cong \Ccal(C,D)$, recovering the hom profunctor on $\Ccal$.
\end{para}


 
\begin{defn}[Collage]
\label{def:collage}
Let $\mathsf{P}: \Fcal\op \times \Ecal \to \Set$ be a profunctor between locally small categories. Recall that the \textbf{collage} of $\mathsf{P}$ is the category $\Coll(\mathsf{P})$ whose collection of objects is $\ob(\Fcal) \sqcup \ob(\Ecal)$ and whose collection of morphisms is $\mor(\Fcal) \sqcup \mor(\Ecal) \sqcup \bigsqcup_{F,E}\mathsf{P}(F,E)$.
\end{defn}

\begin{defn}[Collage diagrams]
\label{def:collagediagram}
Let $\mathsf{P}$ be as in Definition \ref{def:collage}. A \textbf{collage diagram} for $\mathsf{P}$ is simply a diagram in $\Coll(\mathsf{P})$, which we present as an ordinary diagram equipped with a vertical dotted line separating the diagram into a left and right part such that:
\begin{itemize}
    \item To the left of the dotted line are objects and morphisms of $\Fcal$,
    \item To the right of the dotted line are objects and morphisms of $\Ecal$,
    \item Morphisms from an object $F$ on the left to an object $E$ on the right correspond to elements of $\mathsf{P}(F,E)$, and
    \item No morphisms from right to left are permitted.
\end{itemize}
Unless stated otherwise, a diagram implicitly carries the assertion that it \textbf{commutes}, in the sense that each region on each side of the dotted line is a commuting diagram in the respective categories, while a region crossing the dotted line commutes in the sense that applying the relevant components of natural transformations on each side of the line to the elements presented by the arrows crossing the line produces identical results.
\end{defn}

\begin{rem}
\label{rem:collagepres}
Let $\mathsf{P}: \Fcal\op \times \Ecal \to \Set$ be a profunctor between locally small categories, and call $\Coll(\mathsf{P})$ its collage, then we get two canonical functors as described below.
\[\begin{tikzcd}[ampersand replacement=\&]
	\Fcal \&\& \Ecal \\
	\& {\Coll(\mathsf{P})}
	\arrow["{i_{\Ecal}}", curve={height=-12pt}, from=1-3, to=2-2]
	\arrow["{i_{\Fcal}}"', curve={height=12pt}, from=1-1, to=2-2]
\end{tikzcd}\]
It is easy to see that these functors are both fully faithful. Moreover, $i_{\Fcal}$ creates limits and $i_{\Ecal}$ creates colimits.
\end{rem}

While we have given a general definition here, we will exclusively work with collage diagrams for $\ev_{\Ecal}$, viewed as a profunctor $(\pt(\Ecal)\op)\op \times \Ecal \to \Set$, for various choices of $\Ecal$.

\begin{exa}[Visualizing some diagrams in the collage]
\label{exa:presheaves2}
Continuing \Cref{exa:presheaves}, let $\Ecal = [\Ccal\op,\Set]$. Given a commuting square in $\Ccal$ (as in the LHS below), we have a corresponding collage diagram for $\ev_{\Ecal}$ (as in the RHS below):
\[\begin{tikzcd}[ampersand replacement=\&]
	\&\&\&\& \cdot \\
	C \& D \&\& {\yoco(C)} \&\& {\yo(D)} \\
	{C'} \& {D'} \&\& {\yoco(C')} \&\& {\yo(D')} \\
	\&\&\&\& \cdot
	\arrow[""{name=0, anchor=center, inner sep=0}, "{\yoco(f)}"', from=2-4, to=3-4]
	\arrow["x"{pos=0.7}, from=2-4, to=2-6]
	\arrow["{\yo(g)}", from=2-6, to=3-6]
	\arrow["y"'{pos=0.7}, from=3-4, to=3-6]
	\arrow[dotted, no head, from=1-5, to=4-5]
	\arrow["x", from=2-1, to=2-2]
	\arrow["y"', from=3-1, to=3-2]
	\arrow["f"', from=2-1, to=3-1]
	\arrow[""{name=1, anchor=center, inner sep=0}, "g", from=2-2, to=3-2]
	\arrow["\mapsto"{description, pos=0.4}, draw=none, from=1, to=0]
\end{tikzcd}\]
where we write $\yoco(f)$ for the natural transformation between the representable points corresponding to $f$ and by a mild abuse of notation we denote use the same symbol $x$ to denote an element of $\ev(\yoco(C),\yo(D))$ and the corresponding morphism $C \to D$ (and similarly for $y$). In this context, all of the arrows in the collage diagram represent morphisms of essentially the same type, but viewed from different perspectives according to the extension of $\Ccal$ into which they have been promoted.
\end{exa}

\begin{rem}[Alternative handedness]
\label{rem:siteconvention}
We have chosen a site-theoretic perspective on the topos and its category of points. Unfortunately, this is dual to the `semantic' perspective in which $\pt(\Ecal)$ is treated as a category of models of a theory and $\Ecal$ is presented as a subtopos of the category of copresheaves on a category of finitely presentable models. From that perspective, working in the dual of the collages we present would be more intuitive. As such, we caution the reader to take care when working in the left-hand side of our collage diagrams, and invite them to dualize when needed.
\end{rem}


\subsection{Expressing conditions in the collage}

To accustom the reader to reasoning in $\Coll(\ev_{\Ecal})$, we identify how properties of sheaves and points manifest themselves in this setting, beginning with limits and colimits.

\begin{lem}[Lexity]
\label{lem:limpres}
Let $\Ecal$ be a topos and $p \in \pt(\Ecal)$. Let $\Dbb$ be a small category, $F: \Dbb \to \Ecal$ a diagram, $X$ an object of $\Ecal$ and $\lambda: X \Rightarrow F$ a cone over $F$ with apex $X$. Then $p$ sends $\lambda$ to a limit cone if and only if for every cone $x_{-}$ over $F$ with apex $p$ in $\Coll(\ev_{\Ecal})$ there exists a unique factoring element $x^* \in p^*(X)$:
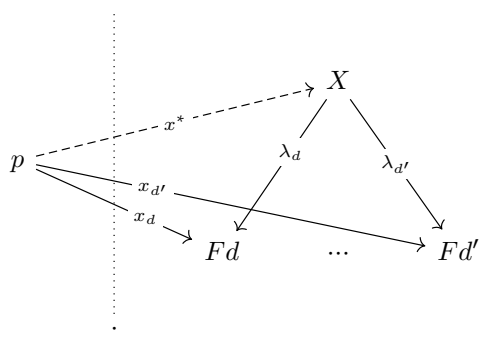
\[\begin{tikzcd}[ampersand replacement=\&]
	\& \cdot \\
	\&\&\& X \\
	p \\
	\&\& Fd \& \cdots \& {Fd'} \\
	\& \cdot
	\arrow["{\lambda_d}"{description, pos=0.4}, from=2-4, to=4-3]
	\arrow["{\lambda_{d'}}"{description}, from=2-4, to=4-5]
	\arrow["x^*"{description}, dashed, from=3-1, to=2-4]
	\arrow["{x_{d'}}"{description, pos=0.3}, from=3-1, to=4-5]
	\arrow[dotted, no head, from=1-2, to=5-2]
	\arrow["{x_d}"{description, pos=0.7}, from=3-1, to=4-3]
\end{tikzcd}\]
\end{lem}
In particular, since points preserve finite limits by definition, this is always true for finite $\Dbb$. It follows that finite limits in $\Ecal$ remain finite limits in the collage $\Coll(\ev_{\Ecal})$. We provide a weaker dual result.

\begin{lem}[Cocontinuity]
\label{lem:colimpres}
Let $\Ecal$, $p$, $\Dbb$, $F$ and $X$ be as in Lemma \ref{lem:limpres}. Let $\gamma: F \Rightarrow X$ be a cone under $F$ with nadir $X$. Then $p^*$ sends $\gamma$ to a jointly epic cone if and only if for any $x \in p^*(X)$ there exists an index $d \in \Dbb$ such that $x$ factors through $\gamma_d$ in the collage diagram:
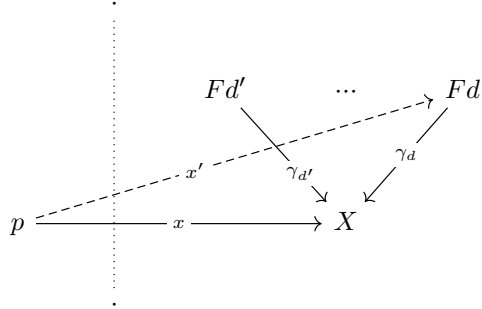
\[\begin{tikzcd}[ampersand replacement=\&]
	\& \cdot \\
	\&\& {Fd'} \& \cdots \& Fd \\
	\\
	p \&\&\& X \\
	\& \cdot
	\arrow["{\gamma_{d'}}"{description, pos=0.7}, from=2-3, to=4-4]
	\arrow["{\gamma_{d}}"{description}, from=2-5, to=4-4]
	\arrow["x"{description}, from=4-1, to=4-4]
	\arrow["{x'}"{description, pos=0.4}, dashed, from=4-1, to=2-5]
	\arrow[dotted, no head, from=1-2, to=5-2]
\end{tikzcd}\]
\end{lem}
For any point $p$ in $\pt(\Ecal)$ we know that $p^*$ preserves small colimits and hence jointly epic families, which means that $p$ is always cone-projective in $\Coll(\ev_{\Ecal})$ for jointly epic families in $\Ecal$.

\begin{cor}[Points as a weak-projectives] \label{cor:coneinj}
Let $(\Ccal,J)$ be a site for the topos $\Ecal$. Then an object $p$ of $\Ind(\Ccal\op)$ restricts to a point of $\Ecal$ if and only if for each $J$-covering sieve $S = \{f_j :C \to C_j\}$ and element $x \in p^*(\yo(C))$, there exists an index $k$ and an element $x' \in p^*(\yo(C_k))$ completing the diagram:
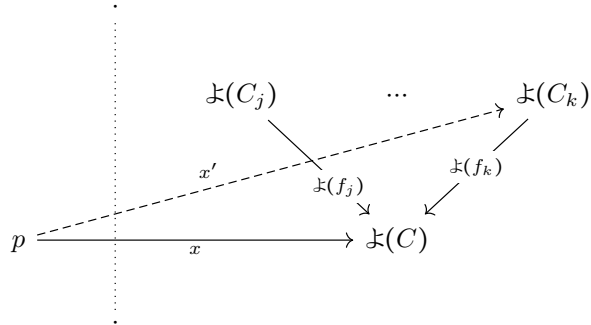
\[\begin{tikzcd}[ampersand replacement=\&]
	\& \cdot \\
	\&\& {\yo(C_j)} \& \cdots \& {\yo(C_k)} \\
	\\
	p \&\&\& {\yo(C)} \\
	\& \cdot
	\arrow["{\yo(f_j)}"{description, pos=0.7}, from=2-3, to=4-4]
	\arrow["{\yo(f_k)}"{description}, from=2-5, to=4-4]
	\arrow["x"', from=4-1, to=4-4]
	\arrow["{x'}"{pos=0.4}, dashed, from=4-1, to=2-5]
	\arrow[dotted, no head, from=1-2, to=5-2]
\end{tikzcd}\]

\end{cor}
\begin{proof}
By Diaconescu's theorem, a point $p$ of $[\Ccal\op,\Set]$ restricts to a point of $\Sh(\Ccal,J)$ if and only if, when $p^*$ is composed with the Yoneda embedding $\Ccal \to [\Ccal\op,\Set]$ to produce a flat functor, the result is $J$-flat, meaning that $J$-covering sieves are sent to jointly epimorphic families. Applying Lemma \ref{lem:colimpres} produces the stated condition.
\end{proof}


However, for greater generality, we can simplify this by considering the image of the covering family, which is to say the sieve viewed as a subobject $S \rightarrowtail \yo(C)$.

\begin{cor}[Points as projectives] \label{crly:injchar}
Let $(\Ccal,J)$ be a site for the topos $\Ecal$. Then an object $p$ of $\Ind(\Ccal\op)$ restricts to a point of $\Ecal$ if and only if for each $J$-covering sieve $S = \{f_j :C \to C_j\}$ and element $x \in p^*(\yo(C))$, there exists a (necessarily unique) element $x' \in p^*(S)$ completing the diagram:
\[\begin{tikzcd}[ampersand replacement=\&]
	\& \cdot \\
	\&\&\& S \\
	p \&\&\& {\yo(C)} \\
	\& \cdot
	\arrow[tail, from=2-4, to=3-4]
	\arrow["x"{description}, from=3-1, to=3-4]
	\arrow["{x'}"{description}, dashed, from=3-1, to=2-4]
	\arrow[dotted, no head, from=1-2, to=4-2]
\end{tikzcd}\]
More generally, for $j:\Fcal \rightarrowtail \Ecal$ a point $p \in \pt(\Ecal)$ factors through $\Fcal$ if and only if for every $j$-dense\footnote{Beware that we conflate the inclusion of topoi and the local operator in $\Ecal$ corresponding to it, since there is little risk of confusion; a monomorphism $m$ is called \textit{$j$-dense} if $j_*j^*m$ is an isomorphism.} monomorphism $m: S \rightarrowtail T$ and element $x \in p^*(X)$, there is a factorization of $x$ through $m$ in $\Coll(\ev_{\Ecal})$:
\[\begin{tikzcd}[ampersand replacement=\&]
	\& \cdot \\
	\&\&\& S \\
	p \&\&\& T \\
	\& \cdot
	\arrow[tail, from=2-4, to=3-4, "m"]
	\arrow["x"{description}, from=3-1, to=3-4]
	\arrow["{x'}"{description}, dashed, from=3-1, to=2-4]
	\arrow[dotted, no head, from=1-2, to=4-2]
\end{tikzcd}\]
\end{cor}

\begin{para}
Given a subtopos $\Fcal \rightarrowtail \Ecal$ the projectivity condition of Corollary \ref{crly:injchar} ensures that we get a fully faithful functor between collages,
\[\Coll(\ev_{\Fcal}) \hookrightarrow \Coll(\ev_{\Ecal}),\]
where fullness on the point side comes from the fact that inclusions of topoi are representably fully faithful in the $2$-category of topoi \cite[Lemma 2.1.3]{Endomodels}.
\end{para}

\section{Improving Deligne} \label{sec:Deligne}

\subsection{Improving points}

Deligne's proof is difficult for several reasons. Our contribution is to decompose it into two layers which can independently be understood and generalized using well-established categorical techniques.

It is straightforward to express the condition of $\Ecal$ having enough points in the context of the collage $\Coll(\ev_{\Ecal})$.

\begin{lem}[Expressing \textit{enough points} in the collage]
\label{lem:manifestenoughpoints}
A topos $\Ecal$ has enough points if and only if for any parallel pair of distinct arrows $f\neq g: X \rightrightarrows Y$ in $\Ecal$ there exists a point $p$ and $z \in p^*(Y)$ such that in the (non-commutative) collage diagram:
\begin{equation}
\label{eq:faithful}
\begin{tikzcd}[ampersand replacement=\&]
	\& \cdot \\
	p \&\& Y \&\& X \\
	\& \cdot
	\arrow["z"{description, pos=0.7}, from=2-1, to=2-3]
	\arrow[dotted, no head, from=1-2, to=3-2]
	\arrow["g"', shift right=2, from=2-3, to=2-5]
	\arrow["f", shift left=2, from=2-3, to=2-5]
\end{tikzcd}    
\end{equation}
the composites $x;f$ and $x;g$ are distinct. In fact, it is sufficient for this to be the case for $X$ restricted to a generating set of objects of $\Ecal$, such as the representable sheaves when $\Ecal = \Sh(\Ccal,J)$. For convenience, from here on we will refer to the composites of the morphisms in \eqref{eq:faithful} as $x=z;f$ and $y=z;g$.
\end{lem}

Suppose we have a topos $\Ecal$ with enough points and a subtopos $j: \Fcal \rightarrowtail \Ecal$. A typical situation will be that $\Ecal$ is a presheaf topos and $\Fcal$ a sheaf topos for a given site, but we work in a site-independent way in this section. Fixing a parallel pair of distinct morphisms of $\Fcal$, there is necessarily a point $p$ of $\Ecal$ which distinguishes their image under $j_*$. In order to construct a point of $\Fcal$ which achieves the same distinction, we iteratively refine $p$ until we attain a point satisfying the projectivity condition of Corollary \ref{crly:injchar}, while ensuring that each successive refinement distinguishes the chosen pair of morphisms. Characterizing a single such procedure brings us to the following definition.

\begin{defn}[Improvement]
\label{def:improves}
Let $\Ecal$ be a topos, $p$ a point of $\Ecal$ and $Z$ an object of $\Ecal$. Suppose we are given $x \neq y$ in $p^*(Z)$. Let $m: S \to T$ be a morphism and $w \in p^*(T)$. An \textbf{improvement of $p$ with respect to the data $(x,y,m,w)$} consists of a point $p'$ of $\Ecal$ and morphisms $v,n$ as in the following (non-commutative) collage diagram:
\begin{equation}
\label{eq:improvement}
\begin{tikzcd}[ampersand replacement=\&]
	\& \cdot \\
	{p'} \&\& S \\
	p \&\& T \\
	\&\& Z \\
	\& \cdot
	\arrow["n"', dashed, from=2-1, to=3-1]
	\arrow["m", from=2-3, to=3-3]
	\arrow[dotted, no head, from=1-2, to=5-2]
	\arrow["w"{description, pos=0.7}, from=3-1, to=3-3]
	\arrow["{w'}"{description, pos=0.7}, dashed, from=2-1, to=2-3]
	\arrow["x"{description, pos=0.7}, shift left, from=3-1, to=4-3]
	\arrow["y"{description, pos=0.7}, shift right=2, from=3-1, to=4-3]
\end{tikzcd}
\end{equation}
such that $v;m = n;w$ but $n;x \neq n;y$. As a shorthand for `$(p,n,v)$ is an improvement of $p$ with respect to $(x,y,m,w)$', we write
\[\impr{(p',n,v)}{p}{(x,y,m,w)}.\]
\textit{We will almost exclusively be concerned with the case where we have an inclusion of topoi $j: \Fcal \rightarrowtail \Ecal$, $Z = j_*(X)$ for some $X \in \Fcal$ and $m$ is a $j$-dense monomorphism.}
\end{defn}

If $p$ is already projective with respect to $m$ then we can take $n$ to be the identity on $p$ and $v$ the factoring morphism of Corollary \ref{crly:injchar}. The following two lemmas are straightforward compositionality results for improvements that will serve us extensively.
\begin{lem}[Compositionality of improvements]
\label{lem:improvementscompose}
Consider the following diagram:
\[\begin{tikzcd}[ampersand replacement=\&]
	\& \cdot \\
	{p''} \&\&\&\& {S'} \\
	{p'} \&\&\& S \\
	p \&\&\& T \& {T'} \\
	\& \cdot \&\& {j_*(X)}
	\arrow["m", from=3-4, to=4-4]
	\arrow["w"{description, pos=0.7}, from=4-1, to=4-4]
	\arrow[dotted, no head, from=1-2, to=5-2]
	\arrow["n"', dashed, from=3-1, to=4-1]
	\arrow["v"{description, pos=0.7}, dashed, from=3-1, to=3-4]
	\arrow["y"{description, pos=0.7}, shift right=2, from=4-1, to=5-4]
	\arrow["x"{description, pos=0.7}, from=4-1, to=5-4]
	\arrow["{m'}", from=2-5, to=4-5]
	\arrow["{w'}"{description}, curve={height=12pt}, from=4-1, to=4-5]
	\arrow["{n'}"', dashed, from=2-1, to=3-1]
	\arrow["{v'}"{description}, dashed, from=2-1, to=2-5]
\end{tikzcd}\]
If $\impr{(p',n,v)}{p}{(x,y,m,w)}$ and $\impr{(p'',n',v')}{p'}{(n;x,n;y,m',n;w')}$ then $p''$ is an improvement of $p$ with respect to both pieces of data, in the sense that both $\impr{(p'',n';n,n';v)}{p}{(x,y,m,w)}$ and $\impr{(p'',n';n,v')}{p}{(x,y,m',w')}$.

Dually, if $\impr{(p',n,v)}{p}{(x,y,m,w)}$ and $\impr{(p'',n',v')}{p'}{(n;x,n;y,m',v)}$:
\[\begin{tikzcd}[ampersand replacement=\&]
	\& \cdot \\
	{p''} \&\&\& {S'} \\
	{p'} \&\&\& S \\
	p \&\&\& T \\
	\& \cdot \&\& {j_*(X)}
	\arrow["m", from=3-4, to=4-4]
	\arrow["w"{description, pos=0.7}, from=4-1, to=4-4]
	\arrow[dotted, no head, from=1-2, to=5-2]
	\arrow["n"', dashed, from=3-1, to=4-1]
	\arrow["v"{description, pos=0.7}, dashed, from=3-1, to=3-4]
	\arrow["y"{description, pos=0.7}, shift right=2, from=4-1, to=5-4]
	\arrow["x"{description, pos=0.7}, from=4-1, to=5-4]
	\arrow["{n'}"', dashed, from=2-1, to=3-1]
	\arrow["{v'}"{description, pos=0.7}, dashed, from=2-1, to=2-4]
	\arrow["{m'}", from=2-4, to=3-4]
\end{tikzcd}\]
then $\impr{(p'',n';n,v')}{p}{(x,y,m';m,w)}$.
\end{lem}

\begin{lem}[Behavior of improvements under base change]
\label{lem:fewercospans}
Consider the following diagram, in which the right-hand square commutes:
\begin{equation}
\label{eq:fewercospans}
\begin{tikzcd}[ampersand replacement=\&]
	\& \cdot \\
	{p'} \&\&\& S \& {S'} \\
	p \&\&\& T \& {T'} \\
	\& \cdot \&\& {j_*(X)}
	\arrow["m", from=2-4, to=3-4]
	\arrow["w"{description, pos=0.7}, from=3-1, to=3-4]
	\arrow[dotted, no head, from=1-2, to=4-2]
	\arrow["n"', dashed, from=2-1, to=3-1]
	\arrow["v"{description, pos=0.7}, dashed, from=2-1, to=2-4]
	\arrow["y"{description, pos=0.7}, shift right=2, from=3-1, to=4-4]
	\arrow["x"{description, pos=0.7}, shift left, from=3-1, to=4-4]
	\arrow["u"{description}, from=3-4, to=3-5]
	\arrow["{m'}", from=2-5, to=3-5]
	\arrow["t"{description}, from=2-4, to=2-5]
\end{tikzcd}
\end{equation}
\begin{itemize}
    \item If $\impr{(p',n,v)}{p}{(x,y,m,w)}$, then $\impr{(p',n,v;t)}{p}{(x,y,m',w;u)}$.
    \item If the right-hand square is a pullback and $v$ is the pullback comparison map, then conversely $\impr{(p',n,v;t)}{p}{(x,y,m',w;u)}$ implies $\impr{(p',n,v)}{p}{(x,y,m,w)}$.
\end{itemize}
\end{lem}

The proofs of Lemmas \ref{lem:improvementscompose} and \ref{lem:fewercospans} are essentially by inspection of Definition \ref{def:improves}, plus in the last instance the fact that pullbacks in $\Ecal$ remain pullbacks in $\Coll(\ev_{\Ecal})$ by Lemma \ref{lem:limpres}.

\begin{thm}[Core theorem]
\label{thm:improvement}
Let $j: \Fcal \rightarrowtail \Ecal$ be an inclusion of topoi. Suppose that for every point $p$ and data $(x,y,m,w)$ as in Definition \ref{def:improves}, there exists an improvement of $p$ with respect to $(x,y,m,w)$. Assuming Zorn's Lemma and that $\Ecal$ has enough points, $\Fcal$ has enough points.
\end{thm}
\begin{proof}
Since $\Ecal$ has enough points, given any parallel pair $f,g:j_*(Y) \rightrightarrows j_*(X)$ there is some point $p$ separating them as in \eqref{eq:faithful}. Fix corresponding elements $x \neq y \in p^*(j_*(X))$. Without loss of generality, we need only consider the case where $m$ is the monomorphism $J \rightarrowtail \Omega$ classifying $j$-dense monomorphisms. 

Using Zorn's Lemma, put a well-ordering on the elements $w \in p^*(\Omega)$; denote the indexing limit ordinal $\kappa$. For each $\alpha \leq \kappa$, we define by induction a point $p_{\alpha}$, a morphism $n_{\alpha}: p_{\alpha} \to p$ in $\pt(\Ecal)\op$ and an element $v_{\alpha} \in p_{\alpha}^*(J)$ which will be an improvement of $p$ with respect to $(x,y,m,w_{\beta})$ for all $\beta < \alpha$
\begin{itemize}
    \item We take $p_0 := p$ and $n_{\alpha} = \id_p$.
    \item For $\alpha = \beta+1$, we take $p_{\alpha}$ to be the improvement of $p_{\beta}$ with respect to $(n_{\beta};x,n_{\beta};y,m,n_{\beta};w_{\beta})$, and $n_{\alpha} = n;n_{\beta}$, where $n$ is the morphism appearing in the improvement diagram \eqref{eq:improvement}, appealing to Lemma \ref{lem:improvementscompose}.
    \item For $\alpha$ a limit ordinal, we first construct the colimit $p_{\alpha}$ in $\pt(\Ecal)$ of the $\{p_{\beta} \mid \beta < \alpha\}$; this remains an improvement of $p$ with respect to $(x,y,m,w_{\beta})$ for all $\beta < \alpha$ because $p_{\alpha}^*(j_*(X))$ and $p_{\alpha}^*(S)$ are computed as the directed colimit of the corresponding values of $p_{\beta}^*$.
\end{itemize}
The result of this procedure, $p_{(1)}:= p_{\kappa}$, is thus an improvement of $p$ with respect to $(x,y,m,w)$ for all $w \in p^*(\Omega)$. However, \textit{a priori} there may be some elements of $p_{(1)}^*(\Omega)$ for which projectivity fails. To resolve this, it suffices to repeat this procedure of iterative improvement to obtain $p_{(2)},p_{(3)},\dotsc$, since all elements of the colimit $p_{(\omega)}$ of this sequence of points will be represented by elements of $p_{(n)}$ for some $n$. This $p_{(\omega)}$ is a point of $\Fcal$ by Corollary \ref{crly:injchar} which distinguishes $f,g$. Since $x,y$ were generic, we conclude that $\Fcal$ has enough points, as required.
\end{proof}

The proof of Theorem \ref{thm:improvement} (rather than its statement) will serve as a foundation for the results in the remainder of the paper. In each proof, the principles involved are the same: demonstrate the existence of a class of improvements using hypotheses on the subtopos or its presentation and iteratively construct points of the subtopos from those improvements. We will not bother to explicitly mention the parallel pair $f,g$ to be distinguished in proofs from here on.

\begin{rem}[Pullbacks provide canonical improvements]
If the pullback of $m$ along $w$ exists in $\Coll(\ev_{\Ecal})$ and \textit{some} improvement of $p$ with respect to $(x,y,m,w)$ exists, then the pullback is automatically an improvement, and in fact the terminal improvement with respect to $(x,y,m,w)$ independently of $x$ and $y$; this is in particular the case when $p$ is projective. Moreover, $n$ will be a monomorphism (and hence an epimorphism in $\pt(\Ecal)$ by Remark \ref{rem:collagepres}). We expect that pullbacks across the gluing boundary in $\Coll(\ev_{\Ecal})$ occur only very rarely, but we can use a weaker variant of this hypothesis to eliminate Zorn's Lemma from Theorem \ref{thm:improvement}.
\end{rem}

\begin{cor}
\label{cor:terminalimprovement}
Let $j:\Fcal \rightarrowtail \Ecal$ be an inclusion of topoi. Suppose that for every point $p$ and data $(x,y,m,w)$ as in Definition \ref{def:improves}, there exists a \textbf{terminal} improvement of $p$ with respect to $(x,y,m,w)$. If $\Ecal$ has enough points, so does $\Fcal$.
\end{cor}
\begin{proof}
With the set-up of $x,y$ and $m$ as in the proof of Theorem \ref{thm:improvement}, consider the directed poset of finite subsets of $p^*(\Omega)$. Given such a subset $W = \{w_1,\dotsc,w_k\}$, consider the corresponding diagram,
\[\begin{tikzcd}[ampersand replacement=\&]
	\& \cdot \\
	{p_W} \&\&\& {J^k} \\
	p \&\&\& {\Omega^k} \\
	\& \cdot \&\& {j_*(X)}
	\arrow["{m^k}", tail, from=2-4, to=3-4]
	\arrow["{\langle w_1,\dotsc,w_k\rangle}"{description, pos=0.7}, from=3-1, to=3-4]
	\arrow[dotted, no head, from=1-2, to=4-2]
	\arrow["{n_W}"', dashed, from=2-1, to=3-1]
	\arrow["{v_W}"{description, pos=0.7}, dashed, from=2-1, to=2-4]
	\arrow["y"{description, pos=0.7}, shift right=2, from=3-1, to=4-4]
	\arrow["x"{description, pos=0.7}, shift left, from=3-1, to=4-4]
\end{tikzcd}\]
where $p_W$ is the terminal improvement of $p$ with respect to the indicated data. Composing with projection morphisms on the right, we deduce that $p_W$ is also an improvement of $w_i$ for each $i$ by Lemma \ref{lem:fewercospans}, and hence we get a canonical morphism $p_W \to p_{W'}$ whenever $W' \subseteq W$. Thus these subsets form a codirected diagram of improvements whose colimit in $\pt(\Ecal)$ we can substitute for $p_{(1)}$ in the proof of Theorem \ref{thm:improvement}. The remainder of the proof continues as before.
\end{proof}

\begin{rem}[Alternative orderings]
We selected the classifying $j$-dense monomorphism for convenience in the proofs of Theorem \ref{thm:improvement} and Corollary \ref{cor:terminalimprovement}, but it is worth noting that the set $p^*(\Omega)$ that we impose a well-ordering on in the proof of Theorem \ref{thm:improvement} already has a natural ordering inherited from the internal order relation on $\Omega$. However, there is no relation \textit{a priori} between the ordering on elements and their respective improvements. At best, it is possible to construct a diagram of improvements indexed by this ordering under the hypotheses of Corollary \ref{cor:terminalimprovement}: for each $w$, we consider the filtered diagram of finite subsets of elements below $w$ and proceed to construct a joint improvement with respect to all elements below $w$ via the construction in the proof of Corollary \ref{cor:terminalimprovement}. It may be that the ordering on $\Omega$ can be better exploited to obtain a constructive result under weaker hypotheses.
\end{rem}

Exploitation of \Cref{cor:terminalimprovement} will depend on a more systematic investigation of the features of collages than we have time or space for here.

\subsection{Making it easier to find improvements}

The existence of improvements with respect to arbitrary data seems demanding. Fortunately, we can relax it in a natural way.
\begin{lem}[Generators suffice]
\label{lem:generatorsuffices}
Let $j: \Fcal \rightarrowtail \Ecal$ be an inclusion of topoi, $p$ a point of $\Ecal$, $x\neq y \in p^*(j_*(X))$. Then $p$ admits an improvement with respect to $(x,y,m,w)$ for all $j$-dense $m:S \rightarrowtail T$ and elements $w \in p^*(T)$ if and only if this is true for $T$ in a generating set of objects of $\Ecal$.
\end{lem}
\begin{proof}
This is an immediate consequence of Lemma \ref{lem:colimpres}: fixing a generating family, every object $T$ of $\Ecal$ admits a covering family of morphisms whose domains are in this family, and any element of $p^*(T)$ must factor through one of the legs of the covering family. An improvement with respect to the factored morphisms is an improvement with respect to the original morphisms by Lemma \ref{lem:fewercospans}.
\end{proof}

\begin{para}[Moving away from the generic $j$-dense morphism]
Examining the proof of Theorem \ref{thm:improvement} in light of Lemma \ref{lem:generatorsuffices}, we find that we can get away with fewer improvements in this way. For instance, suppose $\Ecal = \PSh(\Ccal)$ and $p$ is a representable point (this special case will be the focus of Section \ref{sec:recover}). Then elements of $\yoco(C)^*(\Omega)$ are in bijection with the set of sieves on $C$ in $\Ccal$; if $\Ccal$ is a category with countably many morphisms then there could be uncountably many such elements, but $\coprod_{C,D \in \Ccal} \yoco(C)^*(\yo(D))$ is countable by assumption, and these are sufficient for constructing improvements.
\end{para}

In fact, we can restrict not only the codomains of cospans featuring in improvements, but the class of cospans itself using Lemmas  \ref{lem:fewercospans} and the following definition and lemma which generalize the site-theoretic notions of presieve and multicomposite of sieves.

\begin{defn}[Generating presieve and multicomposites] \label{defn:multicomposite}
Suppose we are given an inclusion $j:\Fcal \to \Ecal$ and $\Tcal$ a generating family of objects for $\Ecal$. A \textbf{generating presieve} for a $j$-dense monomorphism $m:S \rightarrowtail T$ with $T \in \Tcal$ is a collection of morphisms $\{g_i: T_i \to T\}$ where $T_i \in \Tcal$ whose joint image is $m$. A monomorphism $m': S' \rightarrowtail T$ is a \textbf{multicomposite over $m$} if there exists a generating presieve for $m$ as above and $j$-dense monomorphisms $m_i:S_i \rightarrowtail T_i$ such that $m'$ is the joint image of $m_i;g_i$.
\end{defn}

\begin{lem}
For $j:\Fcal \to \Ecal$, $\Tcal$ and $m:S \rightarrowtail T$ as in Definition \ref{defn:multicomposite}, a multicomposite $m'$ over $m$ is $j$-dense.
\end{lem}
\begin{proof}
If $m'$ is a multicomposite over $m$, by definition we have a diagram:
\begin{equation} \label{eqn:multicomposite}
\begin{tikzcd}[ampersand replacement=\&]
	{S_i} \& \cdots \& {S_k} \& {S'} \\
	{T_i} \& \cdots \& {T_k} \\
	\& S \\
	\& T
	\arrow["m", from=3-2, to=4-2]
	\arrow[""{name=0, anchor=center, inner sep=0}, from=2-1, to=3-2]
	\arrow[""{name=1, anchor=center, inner sep=0}, from=2-3, to=3-2]
	\arrow["{m_i}", from=1-1, to=2-1]
	\arrow["{m_k}", from=1-3, to=2-3]
	\arrow["{m'}", curve={height=-18pt}, from=1-4, to=4-2]
	\arrow[curve={height=-12pt}, from=1-1, to=1-4]
	\arrow[from=1-3, to=1-4]
	\arrow[no head, from=0, to=1]
\end{tikzcd}
\end{equation}
where the horizontal line indicates a jointly epimorphic family; since $j^*$ preserves images and epimorphic families and sends $m$ and each $m_i$ to an isomorphism, $j^*(m')$ is the image of a jointly epimorphic family, so is an isomorphism. Thus $m'$ is $j$-dense as claimed.
\end{proof}

\begin{prop}[Minimizing the necessary number of improvements]
\label{prop:lowerboundssuffice}
Let $j: \Fcal \rightarrowtail \Ecal$ be an inclusion of topoi, $p$ a point of $\Ecal$, $x\neq y \in p^*(j_*(X))$. Consider the preorder where:
\begin{itemize}
    \item elements are cospans $(m,w)$ in $\Coll(\ev_{\Ecal})$ whose codomain $T$ is in a generating set of objects of $\Ecal$, where $m: S \rightarrowtail T$ is $j$-dense and $w \in p^*(T)$;
    \item we have $(m,w) \leq (m',w')$ if there exists a pair $(u,t)$ as in \eqref{eq:fewercospans} such that $m;u = t;m'$ and $w' = w;u$.
\end{itemize}
Then for fixed $n:p' \to p$ in $\Coll(\ev_{\Ecal})$, there exists $v$ with $\impr{(p',n,v)}{p}{(x,y,m,w)}$ for all $(m,w)$ if and only if holds for a subset of the $(m,w)$ whose closure under pullbacks and multicomposites provides a lower bound for every element with respect to this ordering.
\end{prop}
\begin{proof}
By the first point of Lemma \ref{lem:fewercospans}, if $(m,w)$ is a lower bound for $(m',w')$ in the preorder, then $\impr{(p',n,v)}{p}{(x,y,m,w)}$ entails $\impr{(p',n,v;t)}{p}{(x,y,m',w')}$, whence a collection of lower bounds is sufficient. By the second point of Lemma \ref{lem:fewercospans} it is sufficient for a pullback of $(m,w)$ to be a lower bound. Finally, given a multicomposite $m'$ of morphisms for which $p$ is projective as in \eqref{eqn:multicomposite} and an element $w \in p^*(T)$ we can construct the required lift along $m'$ by working up the diagram. It is worth noting that we only actually use projectivity with respect to \textit{one} of the $m_i$ in this construction; we will exploit this fact later on.
\end{proof}

Our main use case for Proposition \ref{prop:lowerboundssuffice} will be for minimizing the number of iterations needed in the induction of Theorem \ref{thm:improvement}. The following lemma will (under suitable hypotheses) enable us to avoid the secondary induction procedure in the proof.

\begin{lem}[Folding improvements into a colimit] \label{lem:limitpoint}
Let $\Ecal$ be a topos. Suppose that we are given a point $p \in \pt(\Ecal)$, elements $x \neq y \in p^*(Z)$ and a morphism $m:S \to T$ in $\Ecal$. Suppose that we have constructed a cofiltered diagram $D:\Ibb \to \Coll(\ev_{\Ecal})/p$ in the slice category on objects $n:p' \to p$ such that $n;x \neq n;y$ (prospective improvements of $p$). Suppose that for each $n:p' \to p$ in the diagram and element $w \in {p'}^*(T)$ there is a morphism $n':p'' \to p'$ in the diagram and an element $v \in {p''}^*(S)$ such that $\impr{(p'',n',v)}{p'}{(n;x,n;y,m,w)}$. Then the limit $q$ of this diagram (the colimit in $\pt(\Ecal)$) is projective with respect to $m$.
\end{lem}
\begin{proof}
Every element of $q^*(T)$ is represented by some element $v \in {p'}^*(T)$, and the existence of the improvement $p''$ in the diagram provides the morphism required for projectivity. 
\end{proof}

\section{Recovering completeness results}
\label{sec:recover}

To find improvements in specific cases, we finally need to start exploiting presentations of topoi as sheaves on sites. We develop some auxiliary results for this setting in Section \ref{ssec:presheafprelims} before applying them in Section \ref{ssec:makkairecovery}. We progressively strengthen hypotheses on the site in the remainder of the section, assuming that the site has further limits in Section \ref{ssec:EspKan} and then that it is a frame in \ref{ssec:locales}.

\subsection{Presheaf case preliminaries}
\label{ssec:presheafprelims}

The simplest situation, which will be sufficient for recovering all known completeness results, is to consider an inclusion of topoi of the form $\Sh(\Ccal,J) \to \PSh(\Ccal)$. This situation is particularly convenient: not only does any presheaf topos have enough points, but the most natural class of enough points are the representables, whose inverse image functors preserve all limits. Indeed, all objects of $\pt(\PSh(\Ccal)) \simeq \Ind(\Ccal\op)$ are filtered colimits of representable points, so that properties of $\pt(\PSh(\Ccal))$ can be deduced from those of $\Ccal$, as we demonstrate now.

\begin{prop} \label{indindind}
Let $\Ccal$ be an idempotent-complete small category. The following are equivalent:
\begin{enumerate}
    \item $\Ccal$ has pushouts,
    \item $\Ind(\Ccal)$ has pushouts.
\end{enumerate}
\end{prop}
\begin{proof}
The implication $(2) \Rightarrow (1)$ follows from the fact that when $\Ccal$ is idempotent complete, the Yoneda embedding $\Ccal \to \Ind(\Ccal)$ creates any finite colimits that exist in $\Ind(\Ccal)$, since $\Ccal$ can be recovered as the full subcategory on the finitely presentable objects, which are closed under finite colimits.

To prove $(1) \Rightarrow (2)$, let $\lefthalfcap$ be the walking span. We can present the fact that $\Ccal$ has pushouts via the existence of a left adjoint
\[\colim{}_{\lefthalfcap}: \Ccal^{\lefthalfcap} \to \Ccal\]
to the diagonal functor $\Delta: \Ccal \to \Ccal^{\lefthalfcap}$.

Applying the $\Ind$-completion to this adjunction, we obtain an adjunction between $\Ind(\Ccal)$ and $\Ind(\Ccal^{\lefthalfcap})$. Because the walking span is a well founded category in the sense of \cite[3.1]{henry2023does}, it follows from \cite[Thm. 1.3]{henry2023does} that we have an equivalence of categories, $E: \Ind(\Ccal^{\lefthalfcap}) \xrightarrow{\simeq} \Ind(\Ccal)^{\lefthalfcap}$. Moreover, in the diagram below, it is straightforward to see that the solid triangle commutes. Thus, it follows that $\Delta_{\Ind(\Ccal)}$ has a left adjoint, as required.
\[\begin{tikzcd}[ampersand replacement=\&]
	{\mathsf{Ind(C)}^{\lefthalfcap}} \\
	\\
	{\mathsf{Ind(C^{\lefthalfcap})}} \&\& {\mathsf{Ind}(C)}
	\arrow["{\mathsf{Ind(\Delta_{C})}}", from=3-3, to=3-1]
	\arrow["{\mathsf{Ind(\mathsf{colim}_{C})}}"', shift right, curve={height=18pt}, dashed, from=3-1, to=3-3]
	\arrow["{\Delta_{\mathsf{Ind}(C)}}"', from=3-3, to=1-1]
	\arrow["E", from=3-1, to=1-1]
	\arrow["\top"{marking, allow upside down}, curve={height=18pt}, draw=none, from=3-1, to=3-3]
\end{tikzcd}\]
\end{proof}

\begin{rem}
The forward implication of $\Ccal$ having pushouts implies $\Ind(\Ccal)$ has pushouts does not depend on $\Ccal$ being idempotent complete, and this is the direction we will use in the following; more specifically, we deduce that if $\Ccal$ has pullbacks, then $\pt(\PSh(\Ccal))$ has pushouts. The presence of pullbacks is imposed in the hypotheses of all well-known completeness results; we examine how it can be weakened 
in \Cref{ssec:pb}.

While the idempotent-completeness hypothesis is very mild, in that $\PSh(\Ccal)$ is equivalent to the topos of presheaves on the idempotent-completion of $\Ccal$, we can drop even this condition in the pursuit of sufficient structure to perform the construction of improvements: see \Cref{cor:indfuntamalg}.
\end{rem}

\subsection{Through the looking glass}

The other convenient feature of $\Coll(\ev_{\PSh(\Ccal)})$ is that we can unconditionally slide representables (presheaves and points) across the gluing line. To take advantage of this, when considering a $J$-covering sieve $S$ on $C$, rather than presenting it as a dense monomorphism $S \rightarrowtail \yo(C)$, we instead consider its image under the Yoneda embedding as a collection of morphisms between representables in $\PSh(\Ccal)$. Indeed, this allows us to restrict to a weakly final subdiagram or `generating presieve' for $S$ when needed.

\begin{defn}
\label{def:generatingpresieve}
Let $\Psi$ be a class of diagrams and $\Ccal$ a category. We say a sieve $S$ over $C$ in $\Ccal$ is \textbf{generated by a presieve in $\Psi$} if, when $S$ is viewed as a subcategory of $\Ccal/C$, there exists a weakly final subcategory in $\Psi$. That is, there is a collection of morphisms $\{f_i:C_i \to C\}$ generating $S$ and compatible morphisms between these forming a diagram in $\Psi$.
\end{defn}

The value of this definition is that for any presheaf $F:\Ccal\op \to \Set$ we have,
\[\lim_{(f:C' \to C) \in S} F(C') \hookrightarrow
\lim_{f_i:C_i \to C} F(C_i),\]
where on the right-hand side we compute a limit whose shape is in $\Psi$.

The last ingredient for constructing improvements from these ideas is to focus on points whose inverse image functors preserve limits indexed by these generating presieves.

\begin{prop}[Preserving limits provides improvements]
\label{prop:someimprovement}
Let $\Phi$ be a class of (limit) diagrams. Consider the inclusion $j: \Sh(\Ccal,J) \to \PSh(\Ccal)$ where $\Ccal$ has pullbacks and $J$-covering sieves are generated by presieves in $\Phi\op$. Let $p$ be a point of $\PSh(\Ccal)$ such that $p^*$ preserves $\Phi$-limits and $x \neq y \in p^*(j_*(X))$. Then $p$ has an improvement $p'$ with respect to $(x,y,m,w)$ for all $j$-dense $m:S \rightarrowtail T$ and $w \in p^*(T)$ such that ${p'}^*$ also preserves $\Phi$-limits.
\end{prop}

\begin{rem}[Deligne's blueprint]
While it is more general, the proof of Proposition \ref{prop:someimprovement} essentially follows the strategy of part of Deligne's original proof \cite[Lemma 7.43]{johnstone1977topos} translated into our collage diagram set-up and exploiting the presence of pullbacks more efficiently through Proposition \ref{indindind}. We believe this presentation makes it much easier to follow. From a pedagogical perspective, we consider this to be an important contribution of our paper.
\end{rem}

\begin{proof}
By Lemma \ref{lem:generatorsuffices}, we need only construct improvements with respect to the dense monomorphisms corresponding to $J$-covering sieves. Therefore, suppose we are given a $J$-covering sieve $S = \{f_i:C_i \to C\}$ and an element $w \in p^*(\yo(C))$ (we omit $x,y$ from the diagram):
\[\begin{tikzcd}[ampersand replacement=\&]
	\& \cdot \\
	\&\& {\yo(C_i)} \& \cdots \& {\yo(C_k)} \\
	p \&\&\& {\yo(C)} \\
	\& \cdot
	\arrow["w"'{pos=0.7}, from=3-1, to=3-4]
	\arrow[dotted, no head, from=1-2, to=4-2]
	\arrow["{\yo(f_i)}"{description}, from=2-3, to=3-4]
	\arrow["{\yo(f_k)}"{description}, from=2-5, to=3-4]
\end{tikzcd}\]
Since we are working in $\Coll(\ev_{\PSh(\Ccal)})$, we can translate the representable presheaves across the diagram.
\[\begin{tikzcd}[ampersand replacement=\&]
	\&\&\&\&\& \cdot \\
	\&\& {\yoco(C_i)} \& \cdots \& {\yoco(C_k)} \\
	p \&\&\& {\yoco(C)} \\
	\&\&\&\&\& \cdot
	\arrow["w"', from=3-1, to=3-4]
	\arrow[dotted, no head, from=1-6, to=4-6]
	\arrow["{\yoco(f_i)}"{description}, from=2-3, to=3-4]
	\arrow["{\yoco(f_k)}"{description}, from=2-5, to=3-4]
\end{tikzcd}\]
By \Cref{indindind}, $\pt(\PSh(\Ccal))\op \simeq \mathsf{Ind}(\Ccal\op)\op$ inherits pullbacks from $\Ccal$, so we can construct the pullbacks,
\begin{equation}
\label{eq:coreppullback}
\begin{tikzcd}[ampersand replacement=\&]
	\&\&\&\&\&\& \cdot \\
	{p_i} \& \cdots \& {p_k} \& {\yoco(C_i)} \& \cdots \& {\yoco(C_k)} \\
	\& p \&\&\& {\yoco(C)} \\
	\&\&\&\&\&\& \cdot
	\arrow["w"', from=3-2, to=3-5]
	\arrow[dotted, no head, from=1-7, to=4-7]
	\arrow["{\yoco(f_i)}"{description}, from=2-4, to=3-5]
	\arrow["{\yoco(f_k)}"{description}, from=2-6, to=3-5]
	\arrow["{g_i}"{description}, from=2-1, to=3-2]
	\arrow["{g_k}"{description}, from=2-3, to=3-2]
	\arrow[curve={height=-12pt}, from=2-1, to=2-4]
	\arrow[curve={height=-12pt}, from=2-3, to=2-6]
	\arrow["\lrcorner"{anchor=center, pos=0.125, rotate=45}, draw=none, from=2-1, to=3-5]
	\arrow["\lrcorner"{anchor=center, pos=0.125, rotate=45}, draw=none, from=2-3, to=3-5]
\end{tikzcd}
\end{equation}

Since $j_*(X)$ is a sheaf, for any morphism $v:D \to C$ in $\Ccal$ we have,
\begin{align}
\label{eq:sheafexpand}
\ev(\yoco(D),j_*(X)) &\cong \lim_{f_i \in S\op} \ev(\yoco(\dom(v^*(f_i))),j_*(X)) \\
\nonumber
& \hookrightarrow \lim_{f_i \in {S'}\op} \ev(\yoco(\dom(v^*(f_i))),j_*(X)),
\end{align}
where $S'$ is a generating sieve for $S$ of shape in $\Phi\op$.

Let $\Exp(w)$ be the category of factorizations of $w$ through representables. That is, the category whose objects are pairs of composable morphisms
\[p \xrightarrow[u]{} \yoco(D) \xrightarrow[v]{} \yoco(C)\]
such that $u;v = w$ (we denote this object $(u,\yoco(D),v)$), and whose morphisms are morphisms between the middle objects making the resulting triangles commute. The dual of $\Exp(w)$ is an initial subcategory of the filtered diagram defining $p$, so colimits over it still commute with limits in $\Phi$. By the hypothesis on $S$. Thus we have:
\begin{align}
\nonumber
p^*(j_*(X)) &\cong \colim_{(u,\yoco(D),v) \in \Exp(w)\op} \hspace{3pt} \ev(\yoco(D),j_*(X)) \\
\label{eq:sheaf}
&\cong \colim_{(u,\yoco(D),v) \in \Exp(w)\op} \hspace{3pt} \lim_{f_i \in S\op} \ev(\yoco(\dom(v^*(f_i))),j_*(X))
& \text{ by \eqref{eq:sheafexpand}}\\
\label{eq:inject}
&\hookrightarrow \colim_{(u,\yoco(D),v) \in \Exp(w)\op} \hspace{3pt} \lim_{f_i \in {S'}\op} \ev(\yoco(\dom(v^*(f_i))),j_*(X))\\
\label{eq:commute}
&\cong \lim_{f_i \in {S'}\op} \hspace{3pt} \colim_{(u,\yoco(D),v) \in \Exp(w)\op} \ev(\yoco(\dom(v^*(f_i))),j_*(X)) \\
\nonumber
&\cong  \lim_{f_i \in S} p_i^*(j_*(X)).
\end{align}
where the injection \eqref{eq:inject} is constructed as a colimit of the injections in \eqref{eq:sheafexpand}, and remains injective to the fact that $\Exp(w)$ is filtered, so this colimit preserves monomorphisms.

In particular, there must exist some index $i$ for which $g_i(x) \neq g_i(y)$ in $p_i^*(j_*(X))$, and this $p_i$ is thus the desired improvement. Note that $p_i^*$ preserves $\Phi$-limits because it too is presented by a diagram of shape $\Exp(w)$.
\end{proof}

\begin{rem}[On the sharpness of the assumptions]
\label{rem:relax}
We don't need either of the intermediate steps in \eqref{eq:sheaf} and \eqref{eq:commute} to be bijections; injections would be enough. As such, this argument would still work with:
\begin{itemize}
    \item A $J$-separated presheaf in place of the sheaf $j_*(X)$; as such, when we construct enough points of $\Sh(\Ccal,J)$ using this result we will actually get a set of points which are faithful on the wider class of $J$-separated presheaves. This is not surprising: the sheafification functor is always faithful on separated presheaves, so this observation cannot lead to a stronger result.
    \item $p$ only weakly preserving $\Phi$-limits in the sense that the comparison map is injective. Little enough is known (for the time being) about the sufficient conditions for colimits to weakly commute with limits in this way that it is unclear whether this expands our base of examples.
\end{itemize}
\end{rem}

\begin{rem}[The main special case]
The principal special case of \Cref{prop:someimprovement} is when $\Phi$ is a class of discrete diagrams, so that \eqref{eq:sheafexpand} becomes the injection of the limit appearing in the sheaf condition into a product (the special case of finite products is used in \cite[Lemma 7.43]{johnstone1977topos}). Complementary examples are challenging to come by, since the inverse image functor $p^*$ always preserves finite limits, and for simple candidates of $\Phi$ (such as chains, cofiltered diagrams or connected diagrams of a bounded cardinality), preserving finite limits and $\Phi$-limits is equivalent to preserving finite limits and sufficiently large products.
\end{rem}

\subsection{Deligne and Makkai-Reyes}
\label{ssec:makkairecovery}

By either directly applying the results derived so far or manipulating their proofs, we can recover all known completeness results for topoi. The most straightforward is Deligne's original completeness result, exploiting the special case of Proposition \ref{prop:someimprovement} where $\Phi$ is the class of finite diagrams.

\begin{cor}[Deligne, {\cite[Exposé VI, 9.0]{bourbaki2006theorie}}]
\label{cor:deligne}
Let $(\Ccal,J)$ be a site where $\Ccal$ has pullbacks and $J$-covering sieves are finitely generated (in other words, $\Sh(\Ccal,J)$ is a locally coherent topos). Then assuming Zorn's Lemma, $\Sh(\Ccal,J)$ has enough points.
\end{cor}
\begin{proof}
By Proposition \ref{prop:someimprovement}, since all points preserve finite limits, all improvements exist, so we may apply Theorem \ref{thm:improvement} to the inclusion $\Sh(\Ccal,J) \rightarrowtail \PSh(\Ccal)$ to deduce the existence of enough points of $\Sh(\Ccal,J)$.
\end{proof}

\begin{para}[Not all points are representable!]
\label{para:notallpoints}
As we observed in Section \ref{ssec:presheafprelims}, we can always take the starting collection of enough points of $\PSh(\Ccal)$ to be the representables, whose inverse image functors preserve all small limits. If improvements are constructed using Proposition \ref{prop:someimprovement}, this remains true for the points $p_1,p_2,\dotsc$ constructed in the induction procedure in the proof of Theorem \ref{thm:improvement} (the points $p_i$ constructed in \eqref{eq:coreppullback} will also be representables). The obstacle to greater generality, therefore, is that the point $p_{\omega}$ constructed at the first limit ordinal need not preserve the same limits that the $p_n$ do. There are two solutions to this problem. The first is to constrain the site so that $p_{\omega}$ is the end of the procedure. To do so efficiently, we can take advantage of Proposition \ref{prop:lowerboundssuffice}.
\end{para}

\begin{para}[From topologies to bases]
Consider the special case of Proposition \ref{prop:lowerboundssuffice} where $\Ecal = \PSh(\Ccal)$, $\Fcal = \Sh(\Ccal,J)$ and $p = \yoco(C)$ is a representable point. Then for any $w:C \to D$ and $J$-covering sieve $m:S' \rightarrowtail \yo(D)$, we can construct the diagram:
\[\begin{tikzcd}[ampersand replacement=\&]
	\& \cdot \\
	{p'} \&\&\& S \& {S'} \\
	{\yoco(C)} \&\&\& {\yo(C)} \& {\yo(D)} \\
	\& \cdot \&\& {j_*(X)}
	\arrow["{w^*(m)}"', tail, from=2-4, to=3-4]
	\arrow[Rightarrow, no head, from=3-1, to=3-4]
	\arrow[dotted, no head, from=1-2, to=4-2]
	\arrow["n"', dashed, from=2-1, to=3-1]
	\arrow["v"{description, pos=0.7}, dashed, from=2-1, to=2-4]
	\arrow["y"{description, pos=0.7}, shift right=2, from=3-1, to=4-4]
	\arrow["x"{description, pos=0.7}, shift left, from=3-1, to=4-4]
	\arrow["w"{description}, from=3-4, to=3-5]
	\arrow["m", tail, from=2-5, to=3-5]
	\arrow[from=2-4, to=2-5]
	\arrow["\lrcorner"{anchor=center, pos=0.125}, draw=none, from=2-4, to=3-5]
\end{tikzcd}\]
whence $(\id_C,w^*(m))$ is a lower bound for $(w,m)$ in the preorder of Proposition \ref{prop:lowerboundssuffice}. As such, we can always construct a lower set determined by $J$-covering sieves over $C$, whence the following definition.
\end{para}

\begin{defn}[Cotree and multicomposites]\label{defn:cotree}
For our purposes, a \textbf{cotree}\footnote{We call this a cotree so that our later terminology will be compatible with that of Espindola and Kanalas in \cite{espindola2023every}.} is a poset with a maximal element in which every interval $[x,y] = \{z \mid x \leq z \leq y\}$ is (empty or) isomorphic to a finite total order, as illustrated in the following Hasse diagram:
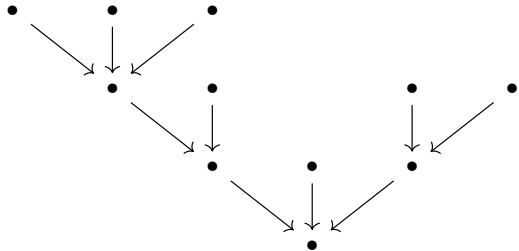
\[\begin{tikzcd}[ampersand replacement=\&]
	\bullet \& \bullet \& \bullet \\
	\& \bullet \& \bullet \&\& \bullet \& \bullet \\
	\&\& \bullet \& \bullet \& \bullet \\
	\&\&\& \bullet
	\arrow[from=3-3, to=4-4]
	\arrow[from=3-4, to=4-4]
	\arrow[from=3-5, to=4-4]
	\arrow[from=2-2, to=3-3]
	\arrow[from=2-3, to=3-3]
	\arrow[from=2-5, to=3-5]
	\arrow[from=2-6, to=3-5]
	\arrow[from=1-1, to=2-2]
	\arrow[from=1-2, to=2-2]
	\arrow[from=1-3, to=2-2]
\end{tikzcd}\]
Note that there is no restriction on the width of the cotree. A \textbf{cotree in a category $\Ccal$} is a diagram indexed by a cotree. For each object in such a diagram, we call the edges entering it its \textit{predecessors}. The \textbf{multicomposite} of a cotree is the presieve obtained by taking the composites of the maximal branches in the cotree, or if the cotree is trivial (of height $0$), consists of the identity morphism on the root.
\end{defn}

\begin{defn}[Warp] \label{defn:basis}
Let $(\Ccal,J)$ be a site. A \textbf{warp} for $J$ is a collection $\Bcal(C)$ of $J$-covering presieves over $C$ for each $C$ such that:
\begin{itemize}
    \item the pullback of a sieve generated by a presieve in $\Bcal$ contains a presieve in $\Bcal$, and
    \item every $J$-covering sieve $S'$ over $C'$ contains a presieve constructed as the multicomposite of a cotree in which the predecessors of each node form a presieve in $\Bcal$.
\end{itemize}
\end{defn}

Note that we can identify the multicomposites appearing in \Cref{defn:basis} with those defined earlier in Definition \ref{defn:multicomposite} by identifying a presieve with the sieve it generates (as a subobject in $\PSh(\Ccal)$).

\begin{rem}[Comparison with pretopologies]
Our Definition \ref{defn:basis} has weaker conditions than usual notions of basis or pretopology (see \cite[Definition III.2.2]{sheavesingeometry}): we do not demand that elements of a warp be closed under pullbacks or multicomposition since closing under these operations recovers $J$. It is hard to claim novelty here. The important thing is that this reduces the number of sieves needed to verify projectivity, which was the goal!
\end{rem}

\begin{defn}[$\omega$-woven sites] \label{defn:woven}
A site $(\Ccal,J)$ is \textbf{$\omega$-woven} if $J$ has a warp $\Bcal$ such that $\Bcal(C)$ is countable at every object $C$.
\end{defn}

\begin{para}[Weaving nomenclature]
The choice of terminology comes from textile production, where the \textit{warp} constitutes the longitudinal threads around which the \textit{weft} is \textit{woven}. We have deliberately selected terms diverging from existing terminology (from topology, for example) to avoid potential confusion.
\end{para}

\begin{thm}
\label{thm:supermakkai}
Let $(\Ccal,J)$ be an $\omega$-woven site where $\Ccal$ has pullbacks. Then assuming the axiom of \textit{dependent countable} choice, $\Sh(\Ccal,J)$ has enough points.
\end{thm}
\begin{proof}
Let $\xi:\omega \times \omega \to \omega$ be a bijection such that $\xi(a,b) \geq a$ for every $b$.

For a given representable point $\yoco(C)$ and $x \neq y \in \yoco(C)(j_*(X))$, we can proceed with the inductive construction in the proof of Theorem \ref{thm:improvement} where improvements are obtained using Proposition \ref{prop:someimprovement}, as follows. Set $\yoco(C_0) := \yoco(C)$ and $n_{0,0} := \id_{\yoco(C)}$. At the $(k+1)$th step, let $(a,b) = \xi^{-1}(k)$ and define $\yoco(C_{k+1})$ and $n_{k+1,k}$ via the improvement diagram:
\[\begin{tikzcd}[ampersand replacement=\&]
	\& \cdot \\
	{\yoco(C_{k+1})} \&\&\& {S_b} \\
	{\yoco(C_k)} \&\&\& {\yo(C_a)} \\
	\& \cdot \&\& {j_*(X)}
	\arrow[dotted, no head, from=1-2, to=4-2]
	\arrow["{n_{k,0} ;y}"{description, pos=0.7}, shift right=2, from=3-1, to=4-4]
	\arrow["{n_{k,0} ;x}"{description, pos=0.7}, shift left, from=3-1, to=4-4]
	\arrow["{n_{k,a}}"{description}, from=3-1, to=3-4]
	\arrow["{m_b}", tail, from=2-4, to=3-4]
	\arrow["v"{description}, dashed, from=2-1, to=2-4]
	\arrow["{n_{k+1,k}}"', dashed, from=2-1, to=3-1]
\end{tikzcd}\]
where $n_{k,a}$ is the composite of $n_{k,k-1};n_{k-1,k-2};\cdots; n_{a+1,a}$ and $m_b$ is the $b$th sieve in the warp of $J$ at $\yo(C_a)$. We need dependent countable choice to complete the construction, since Proposition \ref{prop:someimprovement} requires us to make a choice of improvement at each step.

We now show that this indexing acrobatics ensures that the colimit $p_{\omega}$ of the resulting inductively constructed diagram is projective with respect to all $J$-covering sieves. Suppose we are given a $J$-covering sieve $m':S' \rightarrowtail\yo(C')$ and an element of $p_{\omega}^*(\yo(C'))$; the latter is represented by a morphism $w:\yoco(C_a) \to \yo(C')$. By assumption, the pullback $w^*(m')$ contains a sieve $m$ which can be presented as the multicomposite of a cotree of presieves in $\Bcal$. We proceed by induction on the height of the multicomposite as follows: if the cotree has height $0$, $m$ is the maximal sieve, so projectivity is trivial. Otherwise, suppose that we have proved the result for cotrees of height $h$ and we are given a cotree of height $h+1$. the first layer of the cotree is the $b$th sieve in $\Bcal(C_a)$, whence the diagram contains an improvement ($C_{\gamma}$i say) with respect to $(n_{a,0};x,n_{a,0};y,m_b,\id_{C_a})$. To continue the induction, we take the pullback layer by layer of the subcotree corresponding to the predecessor chosen by this improvement to produce a cotree over $C_{\gamma}$, substituting the first pulled back sieve for one in $\Bcal(C_{\gamma})$ before proceeding to the next layer; this yields a suitable cotree of smaller height over $C_{\gamma}$, so we are done by the induction hypothesis. This is all that is required by Lemma \ref{lem:limitpoint}.
\end{proof}

As a special case, we recover the theorem of Makkai and Reyes regarding so-called `separable topoi'.

\begin{cor}[{\cite[Theorem 6.2.4, page 180]{makkai2006first}}]
\label{cor:makkai}
Let $(\Ccal,J)$ where $\Ccal$ is a countable category with pullbacks and $J$ is generated by a countable family of sieves. Then $\Sh(\Ccal,J)$ has enough points.
\end{cor}
\begin{proof}
It suffices to observe that the collection of all pullbacks of the generating sieves at each object is countable, so forms the countable warp required to apply Theorem \ref{thm:supermakkai}.
\end{proof}


\subsection{Espíndola-Kanalas} \label{ssec:EspKan}
The second of the options alluded to in \ref{para:notallpoints} for finding points which preserve sufficient limits is to construct the diagram in the proof of Theorem \ref{thm:improvement} using limits in $\Ccal$ rather than colimits in $\pt(\Ccal)$. As long as $\Ccal$ has the required limits, we can do this recursively until the diagram of points produced is sufficiently filtered for their colimit in $\pt(\PSh(\Ccal))$ to have inverse image functors which preserve the desired limits. For reasons that we will explain in Example \ref{exa:binaryfan}, this extension requires us to abandon the linear induction, instead constructing a cotree of potential improvements from which we only select an overall improvement after the fact. First we need to generalize the ingredients encountered in the last section.

\begin{defn}[Regular cardinal]
Recall that a cardinal $\kappa$ is \textbf{regular} if it is equal to its own cofinality, which is to say that any chain of ordinals below $\kappa$ whose colimit is $\kappa$ has cardinality $\kappa$.
\end{defn}

For our purposes, the relevant consequences are that:
\begin{itemize}
    \item a $\kappa$-indexed diagram is $\kappa$-filtered, meaning that $\kappa$-indexed colimits commute with limits of size $<\!\kappa$, and
    \item we can construct a bijection $\xi:\kappa \times \kappa \to \kappa$ such that $\xi(\alpha,\beta) \geq \alpha$ for all $\beta$, as we did for the case $\kappa = \omega$.
\end{itemize}

\begin{defn}[$\kappa$-cotree and multicomposites] \label{defn:kappacotree}
For any cardinal $\kappa$, a \textbf{${\kappa}$-cotree} is a poset with a maximal element $1$ in which every interval $[x,y] = \{z \mid x \leq z \leq y\}$ is either empty or isomorphic to the dual of an ordinal $\lambda < \kappa$. The cotrees of Definition \ref{defn:cotree} are $\omega$-cotrees with this convention.

A $\kappa$-cotree in a category $\Ccal$ is a diagram indexed by a $\kappa$-cotree such that the image of a limit ordinal in any branch is the limit of the images of objects below it\footnote{Diagrams satisfying this limit condition are usually called \textit{continuous} diagrams but we drop the adjective since we will not consider more general cotree-indexed diagrams.}.

The \textbf{multicomposite} of a $\kappa$-cotree in $\Ccal$ having limits of $<\!\kappa$-cochains is the sieve over the root constructed from the limits of maximal branches in the cotree.
\end{defn}

\begin{defn}[$\kappa$-warp]
A \textbf{$\kappa$-warp} for a site $(\Ccal,J)$ where $\Ccal$ has limits of $<\!\kappa$-indexed cochains is a collection $\Bcal(C)$ of $J$-covering presieves over $C$ for each $C$ such that:
\begin{itemize}
    \item the pullback of a sieve generated by a presieve in $\Bcal$ contains a presieve in $\Bcal$, and
    \item every $J$-covering sieve $S'$ over $C'$ contains a presieve constructed as the multicomposite of a $\kappa$-cotree in which the predecessors of each node form a presieve in $\Bcal$, and conversely multicomposites of such $\kappa$-cotrees are $J$-covering.
\end{itemize}
\end{defn}

\begin{defn}[$\kappa$-woven sites] \label{defn:kwoven}
A site $(\Ccal,J)$ is \textbf{$\kappa$-woven} if $J$ has a $\kappa$-warp such that $\Bcal(C)$ contains at most $\kappa$ presieves for every $C$.
\end{defn}

The following is a strict improvement of Theorem \ref{thm:supermakkai}, reducing to it in the case $\kappa = \omega$. We take inspiration from the approach of Espíndola and Kanalas in \cite[Thms 2.4, 2.5]{espindola2023every}.

\begin{thm} \label{thm:superespindola}
Let $\kappa$ be a regular cardinal. Let $(\Ccal,J)$ be a $\kappa$-woven site where $\Ccal$ is a category with $\kappa$-small connected limits\footnote{Equivalently with pullbacks and limits of cochains of size $<\!\kappa$.}. Then assuming $\kappa$-dependent choice and the law of excluded middle, $\Sh(\Ccal,J)$ has enough $\kappa$-points (points preserving $\kappa$-small limits).
\end{thm}
\begin{proof}
Let $C_0 \in \Ccal$. We construct a tree of height $\kappa$ in $\Ccal$ with root $C_0$ inductively as follows. Let $\xi: \kappa \times \kappa \to \kappa$ as usual. The predecessors of a node $C$ at height $\gamma^+$ (with $\xi^{-1}(\gamma) = (\alpha,\beta)$) form the $\Bcal(C)$-presieve\footnote{Should we care to complete this construction within the constraints of $\kappa$-dependent choice, we must assume that assignment from the pullback of a $\Bcal$-presieve to a $\Bcal$-presieve contained within it is functorial, hence ``\textit{the} $\Bcal(C)$-presieve''.} contained in $n_{\gamma,\alpha}^*(S_{\beta})$, where $n_{\gamma,\alpha}$ is the (possibly transfinite) composite of the morphisms in the diagram from $C$ to its successor $C'$ at height $\alpha$ and $S_{\beta}$ is the $\beta$th presieve in $\Bcal(C')$. The nodes at a limiting height $\lambda<\kappa$ are the limits in $\Ccal$ of the branches in the tree below them.

The colimit in $\pt(\PSh(\Ccal))$ of any branch of this cotree constitutes a $\kappa$-point of $\Sh(\Ccal,J)$. Indeed, any element of $p^*(\yo(C'))$ is represented by an element $w \in \yoco(C)^*(\yo(C'))$, and by assumption any $J$-covering sieve on $C'$ can be presented as the multicomposite of a ${<\!\kappa}$-cotree in $\Bcal$. We can verify projectivity by inducting up the ${<\!\kappa}$-cotree: supposing that we have shown that there is $\yoco(C_{\beta})$ in the branch which is projective with respect to the subcotree truncated to height $\beta$. This selects a branch in the truncated cotree:
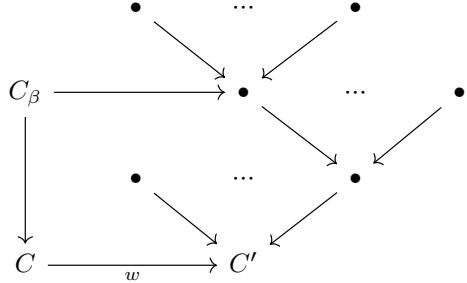
\[\begin{tikzcd}[ampersand replacement=\&]
	\& \bullet \& \cdots \& \bullet \\
	{C_{\beta}} \&\& \bullet \& \cdots \& \bullet \\
	\& \bullet \& \cdots \& \bullet \\
	C \&\& {C'}
	\arrow[from=2-1, to=4-1]
	\arrow[from=3-2, to=4-3]
	\arrow[from=3-4, to=4-3]
	\arrow["w"', from=4-1, to=4-3]
	\arrow[from=2-3, to=3-4]
	\arrow[from=2-5, to=3-4]
	\arrow[from=2-1, to=2-3]
	\arrow[from=1-2, to=2-3]
	\arrow[from=1-4, to=2-3]
\end{tikzcd}\]
We can pull back the $\Bcal$-presieve of predecessors of this node; by construction of the original diagram, there exists $C_{\beta^+}$ in the branch which completes a square with one of the members of this presieve. At a limit ordinal $\lambda$, the limit of the objects $\{C_{\beta}\}_{\beta<\lambda}$ must complete a square with the limit of the corresponding branch of the cotree on the right-hand side, by the universal property of the latter. With this, we are done by Lemma \ref{lem:limitpoint}.

Now given $x \neq y \in \yoco(C_0)(j_*(X))$, by the assumption that the multicomposites of $<\!\kappa$-cotrees are $J$-covering, we conclude that at least one of the branches in the diagram must consist of an improvement of $\yoco(C_0)$ for $(x,y,m,w)$ (for \textit{any} $J$-covering sieve $m$ and element $w$). Otherwise, in each branch of length $\kappa$ there is some $\lambda<\kappa$ such that the composite from height $\lambda$ down to $C_0$ equalizes $x,y$, and these assemble into a $<\!\kappa$-cotree that fails to distinguish $x$ from $y$, contradicting the hypothesis that the multicomposite of such a tree in $\Bcal(C_0)$ must be $J$-covering.
\end{proof}

\begin{exa}[An application of Theorem \ref{thm:superespindola} where Theorem \ref{thm:improvement} fails] \label{exa:binaryfan}
Consider a category $\Ccal$ constructed as follows:
\begin{itemize}
    \item Objects are of the form $C_w$ where $w$ is a finite or $\omega$-indexed binary word, plus a further object $D$;
    \item In addition to identity morphisms, there is a morphism $C_w \to C_{w'}$ whenever $w'$ is a finite prefix of $w$, so the $C_w$ form an $\omega^+$ cotree with branches of length $\omega^+$;
    \item There is a pair of morphisms $C_w \rightrightarrows D$ whenever $w$ is finite, or infinite and eventually constant at $1$, respected by the previous class of morphisms;
    \item There is a single morphism $C_w \to D$ for all other infinite $w$.
\end{itemize}
Part of the structure of this category is sketched here:
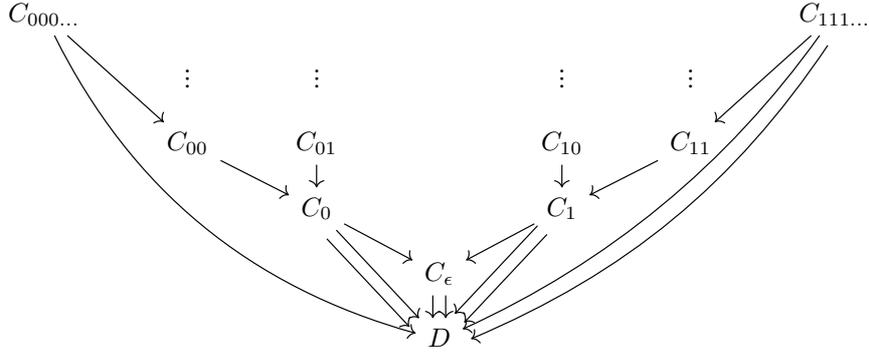
\[\begin{tikzcd}[row sep = small, ampersand replacement=\&]
	{C_{000\dotsc}} \&\&\&\&\&\& {C_{111\dotsc}} \\
	\& \vdots \& \vdots \&\& \vdots \& \vdots \\
	\& {C_{00}} \& {C_{01}} \&\& {C_{10}} \& {C_{11}} \\
	\&\& {C_{0}} \&\& {C_{1}} \\
	\&\&\& {C_{\epsilon}} \\
	\&\&\& D
	\arrow[from=4-3, to=5-4]
	\arrow[from=4-5, to=5-4]
	\arrow[from=3-2, to=4-3]
	\arrow[from=3-3, to=4-3]
	\arrow[from=3-5, to=4-5]
	\arrow[from=3-6, to=4-5]
	\arrow[shift left, from=5-4, to=6-4]
	\arrow[shift right, from=5-4, to=6-4]
	\arrow[from=1-1, to=3-2]
	\arrow[from=4-3, to=6-4]
	\arrow[shift right=2, from=4-3, to=6-4]
	\arrow[curve={height=30pt}, from=1-1, to=6-4]
	\arrow[curve={height=-18pt}, from=1-7, to=6-4]
	\arrow[shift left=2, curve={height=-18pt}, from=1-7, to=6-4]
	\arrow[from=1-7, to=3-6]
	\arrow[shift left, from=4-5, to=6-4]
	\arrow[shift right, from=4-5, to=6-4]
\end{tikzcd}\]
We can equip $\Ccal$ with the Grothendieck topology $J$ having as $\omega_1$-warp $\Bcal$ the pairs $C_{w0} \rightarrow C_w \leftarrow C_{w1}$. This makes the `minimal' sieve of morphisms into $C_{w}$ from objects indexed by infinite words $J$-covering. Moreover, the parallel pair of morphisms $C_{\epsilon} \rightrightarrows D$ remains distinct in $\Sh(\Ccal,J)$, since composition with the morphism from any `eventually 1' objects distinguishes them. If we attempt to construct a point of $\Sh(\Ccal,J)$ starting at a representable point $\yoco(C_w)$ of $\PSh(\Ccal)$ using linear induction as in the proof of Theorem \ref{thm:improvement}, we get stuck at the first limit ordinal, since there is no guarantee that the limit (taken in $\Ccal$) of a cochain of improvements will be an improvement, as witnessed by the leftmost branch starting from any $C_w$. Yet there does exist an improvement in the minimal sieve, namely the morphism coming from $C_{w111\dotsc}$; we know how to find it explicitly here by construction of $\Ccal$, but we relied on the law of excluded middle to construct it in the proof of Theorem \ref{thm:superespindola} so it may not be possible to construct it in general.
\end{exa}

\begin{para}[Returning to induction]
A stronger hypothesis on the category $\Ccal$ (independent of the choice of Grothendieck topology $J$) that allows us to recover the inductive construction is to ask that the canonical map
\begin{equation} \label{eq:cosmall}
\colim_{\alpha<\lambda} \Ccal(C_{\alpha},C) \to \Ccal(\lim_{\alpha<\lambda} C_{\alpha},C)
\end{equation}
is injective for all $C$ and cochains indexed by $\lambda<\kappa$. \Cref{exa:binaryfan} does not satisfy that hypothesis; we leave an elaboration of the proof of the corresponding weaker result to the dedicated reader. However, there is a further convenient consequence of imposing this property.
\end{para}

\begin{cor}[Presheaf collapse] \label{presheaftype}
In the assumptions of Theorem \ref{thm:superespindola}, if $\Ccal$ has limits of chains of length \textit{equal} to $\kappa$ satisfying \eqref{eq:cosmall} then $\Sh(\Ccal,J)$ is equivalent to a presheaf topos.
\end{cor}
\begin{proof}
Under these stronger hypotheses, the point constructed in the proof of Theorem \ref{thm:superespindola} embeds into the representable point obtained by instead taking the limit in $\Ccal$, so we have a representable point which restricts to $\Sh(\Ccal,J)$. This amounts to saying that the corresponding representable presheaves are $J$-sheaves, and hence tiny objects in $\Sh(\Ccal,J)$. That there are \textit{enough} points of this form precisely means that these tiny objects form a generating family for $\Sh(\Ccal,J)$, and thus the topos is of presheaf type. In fact, we can deduce that $\Sh(\Ccal,J) \simeq \PSh(\Ccal')$, where $\Ccal'$ is the full subcategory of $\Ccal$ on objects $C$ such that $\yo(C)$ is a $J$-sheaf.
\end{proof}

\begin{cor}[Espíndola and Kanalas, {\cite[Thm. 5.4]{espindola2023every}}]
\label{cor:espindola}
Let $\kappa$ be a regular cardinal. Let $(\Ccal,J)$ be a site where $\Ccal(C,C')$ has cardinality $\leq\kappa$ for all $C,C'$, $\Ccal$ has $\kappa$-small limits and $J$ is generated under pullbacks and multicomposites of ${<\!\kappa}$-cotrees by at most $\kappa$ sieves. Then $\Sh(\Ccal,J)$ has enough $\kappa$-points.
\end{cor}
\begin{proof}
Just as in Corollary \ref{cor:makkai}, the given data is enough to deduce that we have a $\kappa$-warp of $J$ obtained by taking pullbacks of the generating sieves. Thus we conclude the result by applying Theorem \ref{thm:superespindola}.
\end{proof}

\begin{rem}[Connection to forcing]
As noted by Espíndola and Kanalas in \cite{espindola2023every} and independently pointed out to us by Dianthe Basak, the hypotheses of Theorem \ref{thm:supermakkai} and its proof resemble those of the Rasiowa–Sikorski Lemma. There is a similar resemblance between Theorem \ref{thm:superespindola} and Martin's Axiom, although the hypotheses ensure that we need only appeal to choice in the construction and no such stronger (or logically independent) axiom need be invoked. This suggests a relationship between constructibility of points and extensions of ZFC that we do not have space to explore in depth here.
\end{rem}

\begin{para}[Large cardinals and posetal collapse]
In \cite{espindola2020infinitary}, Espíndola invokes the condition $\kappa^{<\kappa} = \kappa$ on cardinals, which is shorthand for the existence of a bijection $\kappa^{\lambda} \cong \kappa$ for any $\lambda < \kappa$. This condition is implied by inaccessibility. While extensive use is made of this condition, its negation actually has strong consequences if we impose constraints on the size of hom-sets in $\Ccal$. This illustrates how removing size constraints in Theorem \ref{thm:superespindola} may represent a significant improvement of the result of Espíndola and Kanalas, although compelling concrete examples are harder to come by.
\end{para}

\begin{prop}[Posetal collapse for cardinals $\kappa < \kappa^{<\!\kappa}$] \label{prop:posetalcollapse}
Suppose $(\Ccal,J)$ is a site satisfying the conditions of Corollary \ref{cor:espindola} but that there exists $\lambda < \kappa$ such that $\kappa^\lambda > \kappa$. Then $\Ccal$ is a poset.
\end{prop}
\begin{proof}
This is a version of the classic argument of Freyd\footnote{Citation is not endorsement.}, \cite[Chapter 3.D]{freyd1966abelian}. First, since $\kappa$ is not inaccessible, there exists $\lambda'$ with $2^{\lambda'} \geq \kappa$. Given a parallel pair of arrows $f,g:A \rightrightarrows B$ in $\Ccal$, if $f \neq g$ then there must be $2^{\lambda'} \geq \kappa$ morphisms from $A$ to $B' := \prod_{\lambda'} B$ in $\Ccal$. Now taking $\lambda$ as in the statement, we conclude that there are at least $\kappa^\lambda > \kappa$ morphisms from $A$ to $\prod_{\lambda} B'$, a contradiction. Thus we must have had $f=g$, as required.
\end{proof}


\subsection{Spatial locales} \label{ssec:locales}

In the previous sections, we have discussed a number of assumptions on a site $(\Ccal, J)$ ensuring that the topos of sheaves on this site will have enough points. Some of these results have a logical interpretation: we know that coherence is linked to first order logic, and that \textit{separability} is somehow connected to Fourman-Grayson's completeness theorem. \textit{What about their geometric meaning?} This subsection aims to shed some light on this topic, both to provide a more geometric intuition on this theory, and to bring some clarity where the literature seems to have created some confusion. Let us start with this last point.

In this subsection, when we mention a notion from general topology, we shall refer to \cite{munkrestopology,sierpinski2020general, steen1978counterexamples}.

\begin{defn}[Gradations of separability {\cite[Chap. III]{sierpinski2020general}}]
According to the classics, a topological space is:
\begin{enumerate}
    \item separable if it has a countable dense subset,
    \item first countable,  if each point has a countable neighbourhood basis,
    \item second countable, or completely separable, if it has a countable case.
\end{enumerate}

It is straightforward to prove that $3 \Rightarrow 1 \wedge 2$.
\end{defn}


\begin{prop}[A site description of second countability]
Let $X$ be a topological space. The following are equivalent:
\begin{itemize}
    \item $X$ is second countable.
    \item $\Sh(X)$ has a countable (posetal) site with finite limits.
    \item $\Sh(X)$ has a site with countably many objects.
\end{itemize}
\end{prop}
\begin{proof}
If $X$ is second-countable, the \textit{basis} $B$ -- in the classical sense of general topology -- is exactly a lex subposet $B \hookrightarrow \mathcal{O}(X)$ that is dense.

$2\Rightarrow 3$ is immediate. Finally, if $\Sh(X)$ has a countable site, then the supports of representable sheaves form a generating family for $\Sh(X)$, which is a \textit{prebasis} for $X$ (a subposet of $\mathcal{O}(X)$ whose closure under finite intersections is a basis). Since the closure of such a prebasis under finite intersections remains countable, $X$ is second countable.
\end{proof}

\begin{para}[Comparing naming conventions] \label{para:separable?}
Makkai and Reyes justified naming their class of topoi `separable' because of a very special case, namely that the topos of sheaves over a Hausdorff space $X$ is separable if and only if $X$ is a complete separable metric space \cite[Proposition 6.2.3]{makkai2006first}; we find this justification wanting. The classical notion of \textit{separable} or \textit{completely separable} space from general topology are much more general than the (very) special case they are inspired by and this choice creates confusion on how such a geometric object may behave (especially given the hypothesis of Hausdorffness).

In the diagram below we collected all the notions that appeared in this paper or that seem relevant to the discussion.
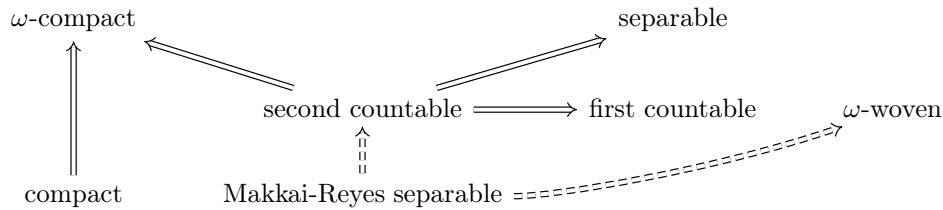
\[\begin{tikzcd}[ampersand replacement=\&]
	{\omega{\text{-compact}}} \&\& {\text{separable}} \\
	\& {\text{second countable}} \& {\text{first countable}} \& {\omega{\text{-woven}}} \\
	{\text{compact}} \& {\text{Makkai-Reyes separable}}
	\arrow[Rightarrow, dashed, from=3-2, to=2-2]
	\arrow[Rightarrow, from=2-2, to=2-3]
	\arrow[curve={height=12pt}, Rightarrow, dashed, from=3-2, to=2-4]
	\arrow[Rightarrow, from=2-2, to=1-1]
	\arrow[Rightarrow, from=3-1, to=1-1]
	\arrow[Rightarrow, from=2-2, to=1-3]
\end{tikzcd}\]

We chose to represent them in different columns because they ostensibly axiomatize different features of the space. Let us comment on them from a site-theoretic perspective. 

\begin{itemize}
    \item compactness and $\omega$-compactness (often called Lindelöf) tell us that covering families have subfamilies of a certain size.
    \item second countability is a bound on the size of the site itself.
    \item woven-ness puts a bound on the complexity of the class of covering families. Notice that this is \textit{a priori} very different from any form of compactness, even though it may seem to have a similar flavour.
    \item Makkai-Reyes separability is a mix of woven-ness and second countability.
    \item separability and first countability are strange notions from a locale-theoretic perspective, because they refer to points.
\end{itemize}

Of course, it happens that these columns do get mixed. For example, second countability is such a strong requirement that it implies $\omega$-compactness, precisely because it cuts down the number of opens quite drastically, and thus forces some compactness-like behaviour into the space. 
\end{para}

We finish this subsection by recalling a result due to Heckmann which is a further consequence of \ref{thm:supermakkai}. Heckmann's proof provides a new insight into the topic via a mild generalization of the Baire category theorem to prove his main theorem \cite[Corollary 3.15]{Heckman}.

\begin{cor}[Heckmann]
A locale $L$ whose frame of opens $\Ocal(L)$ can be presented with countably many relations (inequalities between upward-closed subsets of $\Pcal_{\mathrm{fin}}(G)$) is spatial.
\end{cor}
\begin{proof}
Being presented by countably many relations means that there is an inclusion $j:\Sh(L) \to \PSh(\Pcal_{\mathrm{fin}}(G))$. The relations corresponds to inclusions of subterminal subobjects $\PSh(\Pcal_{\mathrm{fin}}(G))$ which generate the Grothendieck topology on $\Pcal_{\mathrm{fin}}(G)$ corresponding to $j$, and by hypothesis this generating family is countable. For each finite subset $U \in \Pcal_{\mathrm{fin}}(G)$, we obtain a countable warp of covers of $U$ by pulling back these covers; this produces only countably many sieves over $U$ because the subterminal objects form poset. Thus we may apply Theorem \ref{thm:supermakkai}.
\end{proof}

\section{Extending the class of topoi with enough points}
\label{sec:extend}

\subsection{Eliminating pullbacks} 
\label{ssec:pb}

Depending on one's disposition, the requirement that the categories featuring in the sites of the last section should have pullbacks may seem rather demanding. We can eliminate this requirement from Theorem \ref{thm:superespindola} (and hence from Theorem \ref{thm:supermakkai}) by extending a site without pullbacks to an equivalent one having them.

\begin{thm} \label{thm:pbfree}
Let $\kappa$ be a regular cardinal. Let $(\Ccal,J)$ be a $\kappa$-woven site where $\Ccal$ has limits of cochains of size $<\!\kappa$ (but need \emph{not} have pullbacks). Then assuming $\kappa$-dependent choice and the law of excluded middle, $\Sh(\Ccal,J)$ has enough $\kappa$-points.
\end{thm}
\begin{proof}
Let $\overline{\Ccal}$ be the closure of $\Ccal$ in $\PSh(\Ccal)$ under pullbacks. The Grothendieck topology $\overline{J}$ on $\overline{\Ccal}$ such that $\Sh(\overline{\Ccal},\overline{J}) \simeq \Sh(\Ccal,J)$ has a $\kappa$-warp extending the one for $J$ by adding to each of the objects in $\overline{\Ccal}$ the presieve of all morphisms from objects of $\Ccal$ (this generates the maximal sieve for objects in $\Ccal$). It is easily checked that the pullback of such a presieve contains the corresponding presieve and that every $\overline{J}$-covering sieve is the multicomposite of a $\kappa$-cotree in which the first layer is this presieve. As such, we may apply Theorem \ref{thm:superespindola} to deduce the result.
\end{proof}

\begin{para}[Pullbacks resist being eliminated altogether]
It is not possible to eliminate pullbacks quite as directly in Proposition \ref{prop:someimprovement} due to the same obstacle that led us to Theorem \ref{thm:supermakkai}. Namely, as soon as we reach the limit ordinal $\omega$ in the induction of Theorem \ref{thm:improvement}, we lose the ability to improve points with respect to infinite sieves. With some more work, we will be able to improve Deligne's original result, though.
\end{para}

Careful analysis of Deligne's proof leads one to conclude that we only ever use the universal property of pushouts in $\pt(\Ecal)$ to deduce `functorial amalgamation'.

\begin{defn}[Functorial amalgamation]
\label{defn:funtamalg}
We say that $\Ccal$ has \textbf{functorial amalgamation} if there merely exists a functor $\mathsf{amalg}: \Ccal^\lefthalfcap \to \Ccal$ together with a natural transformation $\id_{\Ccal^\lefthalfcap} \Rightarrow  \Delta_{\Ccal} \circ \mathsf{amalg}$.
\end{defn}

Functorial amalgamation is a significant weakening of the requirement of pushouts, where even more than the universal property of the pushout is lost: for example, we do not impose that the amalgamation of a cospan where both maps are identities at an object $C$ returns $C$ itself. Nonetheless, the structure of the proof of Proposition \ref{indindind} can still be applied to deduce the following corollary.    

\begin{cor}[Functorial amalgamation in $\Ind(\Ccal)$]
\label{cor:indfuntamalg}
Let $\Ccal$ be any small category. The following are equivalent:
\begin{enumerate}
    \item $\Ccal$ has functorial amalgamation.
    \item $\Ind(\Ccal)$ has a functorial amalgamation such that $\mathsf{amalg}$ preserves finitely presentable objects.
\end{enumerate}
\end{cor}
\begin{proof}
The proof that $(1) \Rightarrow (2)$ proceeds exactly as it did in Proposition \ref{indindind}.

To deduce $(2) \Rightarrow (1)$, observe that the hypothesis provides functorial amalgamation in the idempotent completion of $\Ccal$ (which is equivalent to the category of finitely presentable objects of $\Ind(\Ccal)$). Every object in the idempotent completion is canonically a retract of one in $\Ccal$, so given the amalgam of a cospan in $\Ccal$ taken in the idempotent completion, we can extend by the section of the retraction map to obtain an amalgam in $\Ccal$.
\end{proof}

\begin{rem}[Eliminating idempotent completeness]
While $\Ind(\Ccal)$ having pushouts only implies that the idempotent completion of $\Ccal$ has pushouts by \Cref{indindind}, we can deduce from \Cref{cor:indfuntamalg} that this still forces $\Ccal$ to have functorial amalgamation.
\end{rem}

\begin{exa}[Adhesiveness provides functorial amalgamation]
\label{exa:CM}
An interesting example of functorial amalgamation is the following. Suppose $\Ccal$ has pushouts and let $\Ccal_{\Mcal} \to \Ccal$ be a wide subcategory (identified with its class of maps $\Mcal \subset \Ccal^\to$) such that $\Mcal$ is stable under pushouts. This means precisely that in the diagram below, we can construct the dashed functor, which will automatically be a functorial amalgamation. Yet, in full generality this construction will not provide pushouts for $\Ccal_{\Mcal}$, since we do not demand that the pushout comparison map belongs to $\Mcal$.
\[\begin{tikzcd}[ampersand replacement=\&]
	{\Ccal_{\Mcal}^{\lefthalfcap}} \& {\Ccal^{\lefthalfcap}} \\
	{\Ccal_{\Mcal}} \& \Ccal
	\arrow["i"', from=2-1, to=2-2]
	\arrow[from=1-2, to=2-2]
	\arrow["{i^{\lefthalfcap}}", from=1-1, to=1-2]
	\arrow[dashed, from=1-1, to=2-1]
\end{tikzcd}\]

Some concrete instances of this construction are as follows:
\begin{enumerate}
    \item The simplest is given the case where $\Ccal$ is $\mathsf{Fin}$, the category of finite sets, and $\Ccal_\Mcal$ is its wide subcategory of monomorphisms. The induced amalgamation is not a pushout, since (using $n$ to denote an $n$-element set), the dashed arrow in the following diagram in $\mathsf{Fin}$ is not a monomorphism.
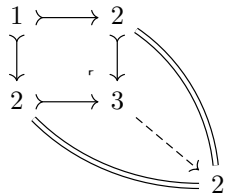
\[\begin{tikzcd}[ampersand replacement=\&]
	1 \& 2 \\
	2 \& 3 \\
	\&\& 2
	\arrow[tail, from=2-1, to=2-2]
	\arrow[tail, from=1-2, to=2-2]
	\arrow[tail, from=1-1, to=1-2]
	\arrow[tail, from=1-1, to=2-1]
	\arrow["\lrcorner"{anchor=center, pos=0.125, rotate=180}, draw=none, from=2-2, to=1-1]
	\arrow[curve={height=-12pt}, Rightarrow, no head, from=1-2, to=3-3]
	\arrow[curve={height=12pt}, Rightarrow, no head, from=2-1, to=3-3]
	\arrow[dashed, from=2-2, to=3-3]
\end{tikzcd}\]
    
    \item More generally, monomorphisms are stable under pushout in any \textit{adhesive category} by definition, so we can perform this construction. 
\end{enumerate}
\end{exa}





\begin{defn}[Functorial coamalgamation] \label{defn:functcoamalgam}
We say that a category $\Ccal$ has \textbf{functorial coamalgamation} if $\Ccal\op$ has functorial amalgamation. We say that a site $(\Ccal,J)$ has \textbf{functorial coamalgamation} if this is true for $\Ccal$ \textit{and} for a $J$-covering presieve $\{f_i:C_i \to C\}_{i \in I}$ and any morphism $h:D \to C$, the coamalgamation squares:
\begin{equation}  \label{eq:coamalgams}  
\begin{tikzcd}[ampersand replacement=\&]
	{D_i} \& {C_i} \\
	D \& C
	\arrow["h"', from=2-1, to=2-2]
	\arrow["{f_i}", from=1-2, to=2-2]
	\arrow["{h_i}", from=1-1, to=1-2]
	\arrow["{g_i}"', from=1-1, to=2-1]
\end{tikzcd}
\end{equation}
are such that $\{g_i:D_i \to D\}_{i \in I}$ forms a $J$-covering presieve. (If the coamalgamation squares are pullbacks, this is automatic).
\end{defn}

\begin{prop}[Constructing improvements without pullbacks]
\label{prop:functcoamalgimprove}
Let $\Phi$ be a class of limit diagrams. Consider the inclusion $j: \Sh(\Ccal,J) \to \PSh(\Ccal)$ where $(\Ccal,J)$ is a site with functorial coamalgamation in the sense of Definition \ref{defn:functcoamalgam} and $J$-covering sieves are generated by presieves in $\Phi\op$. Let $p$ be a point of $\PSh(\Ccal)$ such that $p^*$ preserves $\Phi$-limits and $x \neq y \in p^*(j_*(X))$. Then $p$ has an improvement $p'$ with respect to $(x,y,m,w)$ for all $j$-dense $m:S \rightarrowtail T$ and $w \in p^*(T)$ such that ${p'}^*$ also preserves $\Phi$-limits.
\end{prop}
\begin{proof}
Checking the proof of Proposition \ref{prop:someimprovement}, we see that everything works identically for a site with functorial coamalgamation after replacing the pullbacks $v^*(f_i)$ with the coamalgams $g_i$ of \eqref{eq:coamalgams}.
\end{proof}



\begin{cor}[Deligne without pullbacks] \label{cor:funamalgversion}
Let $(\Ccal,J)$ be a site with functorial coamalgamation in which every $J$-covering sieve contains a finite covering presieve; this corresponds to taking $\Phi$ to be the class of finite diagrams in Proposition \ref{prop:functcoamalgimprove}. Then $\Sh(\Ccal,J)$ has enough points.
\end{cor}

\begin{rem}[Not all locally finitely presentable topoi]
The topoi which admit a site $(\Ccal,J)$ such that every $J$-covering sieve contains a finite covering presieve are exactly the \emph{locally finitely generated} topoi (or \textit{compactly generated} topoi in the terminology of \cite{rogers2021supercompactly}): those which have a generating set of finitely generated (also called `quasi-compact') objects. The class of locally finitely generated topoi notably includes the locally finitely presentable ones.

We cannot hope to completely remove functorial coamalgamation from \Cref{cor:funamalgversion}, since any atomic topos is locally finitely generated, and there are well-known examples of atomic topoi without points, such as the classifying topos for the theory of uncountable complete binary trees due to Malitz, see \cite[Example D3.4.14]{elephant2}.

On the other hand, this does not \textit{a priori} determine whether or not there exists a locally finitely presentable topos without enough points (assuming the axiom of choice). Since the release of the original preprint of the present article, a counterexample showing that not all locally finitely presentable topoi have enough points has been provided by Jérémie Marquès: see \Cref{rem:Marques}.
\end{rem}

\begin{rem}
\label{rem:Marques}
Since the preprint version of the present paper appeared, Jérémie Marquès claims to have shown that the example of an atomic topos with no points \cite[Example D3.4.14]{elephant2} is actually locally finitely presentable \cite{jeremy}. It would follow that being locally finitely presentable is not sufficient to admit \textit{any} points at all, even assuming the axiom of choice!
\end{rem}

In spite of this, we can extend Corollary \ref{cor:deligne} at least a little by constraining the local size of the site. 

\begin{cor} \label{cor:lfpcountable}
A locally finitely presentable topos whose subcategory of finitely presentable objects has essentially countable slice categories has enough points.
\end{cor}
\begin{proof}
Locally finitely presentable toposes are precisely those admitting a site $(\Ccal,J)$ in which $J$ is generated by sieves having finite final presieves\footnote{This is essentially the content of \cite[Prop. 5.5]{di2020gabriel}, although the condition of finality of the presieve was forgotten in the formulation of the theorem there.}, which corresponds to taking $\Phi$ to be the class of finite diagrams in Proposition \ref{prop:functcoamalgimprove} but without assuming functorial coamalgamation. The stated hypothesis ensures that we can choose a site in which there are at most countably many morphisms with a given object as codomain. This means that there are at most countably many finite $J$-covering families over each object, whence the collection of all such is an $\omega$-warp for $J$ and $(\Ccal,J)$ satisfies the hypotheses of \Cref{thm:pbfree}. 
\end{proof}

\begin{rem}[Not quite Hofmann-Lawson]
It is worth noting that if we specialize Corollary \ref{cor:lfpcountable} to the localic case, we do not directly recover the well-established result that continuous locales have enough points, proved by Hofmann and Lawson \cite{hofmann1978spectral}. While parts of the proof have a similar flavour to improvements, they critically use Zorn's lemma in a way that cannot be directly generalized beyond frames, so cannot be presented as a variant of Deligne's strategy. 
\end{rem}

\begin{para}[Morphisms and comorphisms of sites]
Of course, if a topos $\Gcal$ admits a geometric surjection from a topos with enough points, then it too has enough points. From a site-theoretic perspective, we can construct geometric morphisms using either morphisms or comorphisms of sites. Caramello established necessary and sufficient conditions for when the resulting geometric morphisms are surjections \cite{caramello2019denseness}. In principle one could characterize the sites $(\Dcal,J)$ admitting a surjection-inducing morphism of sites to (resp.\ a surjection-inducing comorphism of sites from) a site $(\Ccal,K)$ satisfying the conditions of Theorem \ref{thm:pbfree} or Corollary \ref{cor:funamalgversion}. Lacking compelling examples, we leave that effort to future work.
\end{para}

\subsection{Zooming out: the 2-category of topoi}\label{sec:promenade}

There is a more abstract perspective on the construction performed in the proof of Theorem \ref{thm:improvement}. We considered a parallel pair of morphisms $f,g:Y\rightrightarrows X$ in $\Fcal$ and a point $p$ such that $p^*j_*f \neq p^*j_*g$; from these we constructed a point $p'$ of $\Fcal$ equipped with a morphism $p \Rightarrow p';j$ in $\pt(\Ecal)$:
\[\begin{tikzcd}[ampersand replacement=\&]
	\& \Fcal \\
	{\Set} \& \Ecal.
	\arrow["f", from=1-2, to=2-2]
	\arrow["p"', from=2-1, to=2-2]
	\arrow[""{name=0, anchor=center, inner sep=0}, "{p'}", dashed, from=2-1, to=1-2]
	\arrow[shorten >=2pt, Rightarrow, from=2-2, to=0]
\end{tikzcd}\]
However, we can more properly summarize the result by presenting a collection of `enough' points of $\Ecal$ as a geometric morphism from $\Set^K$ for some set $K$. Thus we construct a square:
\[\begin{tikzcd}[ampersand replacement=\&]
	{\Set^{K'}} \& \Fcal \\
	{\Set^K} \& \Ecal
	\arrow["f", from=1-2, to=2-2]
	\arrow["P"', from=2-1, to=2-2]
	\arrow["{P'}", from=1-1, to=1-2]
	\arrow["{\Set^{\pi}}"', from=1-1, to=2-1]
	\arrow[Rightarrow, from=2-1, to=1-2]
\end{tikzcd}\]
where on the left $\pi$ is the map sending each point constructed in Theorem \ref{thm:improvement} to the one it was constructed from. In general we will need many points $p'$ for each point $p$, but under some circumstances at most one is needed, as the following result shows.

\begin{thm}
\label{thm:closedinc}
The class of topoi with enough points is closed under closed inclusions. That is, letting $f: \Fcal \to \Ecal$ be a closed inclusion such that $\Ecal$ has enough points, it follows that $\Fcal$ has enough points.
\end{thm}
\begin{proof}
Observe that surjections are stable by pullback along closed inclusions. Indeed, consider a pullback square of geometric morphisms
\[\begin{tikzcd}
	\Gcal \ar[r, "q"] \ar[d, "g"'] \ar[dr, phantom, "\lrcorner", very near start] & \Fcal \ar[d, "f"] \\
	\Scal \arrow[r, "p"'] & \Ecal,
\end{tikzcd}\]
in which $p$ is a surjection and $f$ is a proper inclusion. By Theorem \cite[I.6.1]{moerdijk2000proper}, the Beck-Chevalley condition $p^*f_* \cong g_*q^*$ holds. Since $p^*$ and $f_*$ are faithful by assumption, this forces $q^*$ to be faithful, and hence $q$ is a surjection. Moreover, closed inclusions are stable under pullback, so $g$ is a closed inclusion too. 
In particular, considering the case where $\Ecal$ has enough points, we may take $\Scal$ to be $\Set^K$ for $K$ a sufficient set of points, at which point $\Fcal$ admits a surjection from a (closed) subtopos of $\Set^K$, which is necessarily of the form $\Set^{K'}$ for some subset $K' \subseteq K$. Hence $\Fcal$ has enough points.
\end{proof}

\begin{rem}[Johnstone proved the localic case]
One might compare the theorem with \cite[C1.2.6(b)]{Sketches}, which is the localic case of this result. Johnstone remarks there that closed sublocales of spatial locales being spatial is a non-constructive result, and it is worth isolating the non-constructive aspects of our proof, which lies in the fact that a subtopos of the copower $\Set^K$ is necessarily of the form $\Set^{K'}$. For a more generic topos $\Scal$, a closed subtopos of $\Scal^K$ will instead be a $K$-indexed coproduct of closed subtopoi of $\Scal$. Thus we arrive at the following more general result. 
\end{rem}

\begin{cor}[Relative version]
Let $\Scal$ be any base topos. Given a topos $\Ecal$ over $\Scal$ which has enough $\Scal$-valued points, any closed subtopos $\Fcal$ of $\Ecal$ has enough ($\Scal$-valued) `closed sub-points', meaning that the collection of geometric morphisms coming from closed subtopoi of $\Scal$ have jointly conservative inverse image functors.
\end{cor}

\begin{para}[Beck-Chevalley is the key]
The essential feature exploited in the proof of Theorem \ref{thm:closedinc} is the Beck-Chevalley condition. Indeed, we only needed the \textit{weak} Beck-Chevalley condition, where we merely have a componentwise monic transformation, which amounts to $g$ being \textbf{proper} by \cite[Corollary C3.2.22]{elephant2}. For subtopoi this adds nothing to Theorem \ref{thm:closedinc}, however, since an inclusion is proper if and only if it is tidy, if and only if it is closed. 

We can also attain the weak Beck-Chevalley condition by imposing dual hypotheses, but we shall see that this doesn't tell us quite so much.
\end{para}

\begin{lem}[Open points]
Suppose that $\Ecal$ has enough \textit{open} points, which amounts to the existence of an open surjection $\Set^K \to \Ecal$ for some set $K$. Then any subtopos of $\Ecal$ also has enough open points.
\end{lem}
\begin{proof}
The argument proceeds identically to Theorem \ref{thm:closedinc}, this time deducing the weak Beck-Chevalley condition from the fact that the geometric morphism $\Scal \to \Ecal$ is open by \cite[Theorem C3.1.27]{elephant2}. We again rely on the fact that every subtopos of $\Set^K$ is of the form $\Set^{K'}$.
\end{proof}

This result is redundant since by \cite[Lemmas C3.5.1, C3.5.2]{elephant2} a topos admitting an open surjection from $\Set^K$ is necessarily Boolean, so all of its subtopoi are closed and we may thus apply Theorem \ref{thm:closedinc} to them. Its extension to general topoi may be non-trivial, however.

\begin{cor}
Let $\Scal$ be any base topos. Given a topos $\Ecal$ over $\Scal$ which has enough open $\Scal$-valued points, any subtopos $\Fcal$ of $\Ecal$ has enough ($\Scal$-valued) sub-points, meaning that the collection of geometric morphisms coming from subtopoi of $\Scal$ have jointly conservative inverse image functors (and in fact the sub-points which are open suffice).
\end{cor}

\begin{rem}[Not quite Deligne]
\label{rem:notDeligne}
Theorem \ref{thm:closedinc} is not as general as it might seem at first sight. Indeed, the inclusion of a coherent topos into the category of presheaves on its coherent objects is seldom a closed inclusion, 
although it preserves filtered colimits (and so is \textit{relatively tidy} in the sense of Moerdijk and Vermeulen, \cite{moerdijk2000proper}). Conversely, since we have not been able to verify that every locally finitely presentable topos has enough points, we do not know whether the relativised weak Beck-Chevalley condition that applies to relatively tidy morphisms via the construction of the `comma topos' is sufficient to witness the existence of enough points of relatively tidy subtopoi more generally. 
\end{rem}

\begin{rem}[A logical interpretation]
Theorem \ref{thm:closedinc} is easier to interpret than Deligne's theorem from a logical point of view. Following \cite[4.2.2.2]{caramello2018theories}, let $\mathbb{E}$ be a geometric theory with enough set models\footnote{Which means that $\mathbb{E} \vdash \phi$ if and only if $\phi$ is true in all $\Set$-valued models.}, and let $\psi$ be a geometric propositional formula (a geometric sequent with no free variables formed from finitary conjunctions, arbitrary disjunctions and existential quantification). Then the quotient theory $\mathbb{E} \cup (\psi \vdash \bot)$ has enough $\Set$-models too. In other words, adding a finite number of negated propositional formulas to a geometric theory cannot break the existence of enough models.
\end{rem}



\section*{Acknowledgements}
The authors started working on this paper in late 2021, and this project has been a collection of ups and downs, attempts and failures. We are indebted to our friends and colleagues \textit{Axel Osmond} and \textit{Joshua Wrigley} who have encouraged us throughout the journey, as well as colleagues in our respective institutions who have patiently listened to vague outlines of the work during its long gestation. A significant portion of this research was carried out during two invitations: Ivan Di Liberti was invited by Université Sorbonne Paris Nord (Paris) in 2022 and Morgan Rogers was invited by Stockholm University (Stockholm) in 2023. Both authors are grateful to these institutions for the financial support that made the research visits possible.

\bibliography{thebib}
\bibliographystyle{alpha}

\end{document}